\documentclass[11pt,letterpaper]{article}
\usepackage{mathrsfs}
\usepackage{geometry}

\usepackage{graphicx}
\usepackage{amssymb}
\usepackage{epstopdf}
\usepackage{amsmath,amscd}
\usepackage{amsthm}
\theoremstyle{plain} \textwidth=430pt \textheight=650pt
\makeatletter

\newcommand{\Rmnum}[1]{\expandafter\@slowroman #1@}
\makeatother

\newtheorem{theorem}{Theorem}

\newtheorem{lemma}[theorem]{Lemma}

\newtheorem{assumption}[theorem]{Assumption}

\title{Conformal invariance of dimer heights on isoradial double graphs}
\author{Zhongyang Li}
\date{}

\begin{document}
\maketitle

\begin{abstract}

An isoradial graph is a planar graph in which each face is inscribable into a circle of common radius. We study the 2-dimensional perfect matchings on a bipartite isoradial graph, obtained from the union of an isoradial graph and its interior dual graph. Using the isoradial graph to approximate a simply-connected domain bounded by a simple closed curve, by letting the mesh size go to zero, we prove that in the scaling limit, the distribution of height is conformally invariant and converges to a Gaussian free field.

\end{abstract}

\section{Introduction}

One of the most important starting points for two-dimensional critical lattice models in statistical physics is the assumption that the continuous limit is universal and conformally invariant. It says that, in the limit when the lattice spacing $\epsilon$ goes to zero, macroscopic quantities of the model transform covariantly under conformal maps of the domain (conformal invariance) and are independent of the lattice (universality). Under the assumption of universality and conformal invariance, physicists have successfully predicted exact values of certain critical exponents. However,  the conformal invariance and universality assumptions were beyond the mathematical justification until very recently \cite{RK00, RK002,SS01,LSW04,SS09,CS12,RK11}.

In this paper, we focus on perfect matchings on planar graphs. A perfect matching, or dimer covering of a finite graph is a set of edges covering all the vertices exactly once. The dimer model is the study of random dimer coverings of a graph. The dimer model has been known to be integrable since the work \cite{pk,el} and hence amenable to a number of techniques, see \cite{rb,aljmp}. Following Kasteleyn's technique for dimers on planar graphs \cite{pk}, and by assigning weights to edges, one can define a probability measure on random dimer coverings \cite{RK97}. The height function is a random function  which assigns a unique number to each faces of the graph for each realization of random perfect matchings. If the underlying graph is a subgraph of the square grid $\mathbb{Z}^2$, Kenyon \cite{RK00,RK01} proved the conformal invariance of height distribution in the scaling limit under certain boundary conditions, given the uniform measure of random perfect matchings. The hexagonal lattice case were studied in \cite{RS04,RK08}.

An isoradial graph is a planar graph in which each face is inscribable in a circle of common radius. It was introduced by Duffin \cite{RD68} in the late sixties, in an equivalent form of rhombic lattices, and reappeared recently in the work \cite{CM01,RK02,CS11,CS12,BD11, JD11, GM}. Isoradial graphs are a  large class of graphs where the complex analysis techniques have a ``nice'' discrete analog, and hence a natural setting for the universality assumption, which includes, and is more general than the class of regular graphs, see \cite{CS11} for an exposition of the discrete complex analysis technique on isoradial graphs. Ising models on isoradial graphs have been studied extensively in \cite{CM01,BD11,CS12}, and spectacular results were proved including conformal invariance and universality \cite{CS12}. For perfect matchings on isoradial graphs, Kenyon \cite{RK02} proved an explicit form for the inverse of the weighted adjacency matrix of bipartite isoradial graphs on the whole plane. Following that, it is proved in \cite{BdT07}, that the height function of perfect matchings on the whole plane converges to a Gaussian free field, yet the conformal invariance and boundary conditions were not addressed in \cite{BdT07} since the paper deals with graphs on the whole plane.

In this paper, we use isoradial graphs to approximate an arbitrary simply connected domain in the plane bounded by a simple closed curve, and prove the following results

\begin{theorem}Let $\Omega$ be a simply connected bounded domain in the plane bounded by a simple closed curve. Assume $\partial \Omega$ has a straight portion $L_0$. Let $\mathcal{G}_{\delta}$ be an isoradial graph embedded into the whole plane with common radius $\delta$, and $\mathcal{G}_{\Omega,\delta}$ be the largest subgraph of $\mathcal{G}_{\delta}$ consisting of faces, and completely inside $\Omega$.  Let $\mathcal{G}'_{\Omega,\delta}$ be the interior dual graph of $\mathcal{G}_{\Omega,\delta}$, and assume $\mathcal{G}_{\Omega,\delta}$ is the interior dual graph of $\mathcal{G}''_{\Omega,\delta}$. Assume $\partial\mathcal{G}_{\Omega,\delta}''$ has a straight part $L_{0,\delta}$ approximating $L_0$.  Let $\mathcal{G}_{\Omega,\delta}^D$ be the superposition of $\mathcal{G}_{\Omega,\delta}$ and $\mathcal{G}'_{\Omega,\delta}$ with one boundary vertex of $\mathcal{G}_{\Omega,\delta}$, approximating a point on $L_0$ as $\delta\rightarrow 0$, removed. As the mesh size $\delta\rightarrow 0$, the distribution of the height function for random perfect matchings on $\mathcal{G}_{\Omega,\delta}^D$, in the scaling limit, is conformally invariant and converges to a Gaussian free field.
\end{theorem}

Here the ``interior dual graph'' of $\mathcal{G}_{\Omega,\delta}$, is the subgraph of the infinite dual graph $\mathcal{G}_{\delta}'$ of $\mathcal{G}_{\delta}$, such that the vertices of the ``interior dual graph'' are in one-to-one correspondence with the faces of $\mathcal{G}_{\Omega,\delta}$, and each edge of the ``interior dual graph'' is an edge connecting two vertices of the ``interior dual graph'' and is a dual edge of an edge of $\mathcal{G}_{\Omega,\delta}$. In Figure \ref{isorum}, the graph with solid black edges is the primal graph $\mathcal{G}_{\Omega,\delta}$, and the dual graph bounded by the blue edges is the ``interior dual graph'' of $\mathcal{G}_{\Omega,\delta}$. Moreover, $\mathcal{G}_{\Omega,\delta}$ is the interior dual graph of the graph bounded by the outer red edges, with edges given by red lines or blue lines.

In particular, we do not assume the graph to be periodic, and we do not assume the
boundary to be smooth except for a straight portion. The problem solved in this paper
was also mentioned in \cite{CS12}, namely, whether the conformal invariance result for domino
tilings is also true for isoradial graphs. We construct a ``nice" approximation of discrete
isoradial graphs to continuous domains, so that the boundary conditions could be dealt
with and the convergence result follows. Since the bipartite isoradial graph is obtained by
the superposition of an arbitrary isoradial graph and its interior dual graph, the boundary
conditions are discrete analogues of the Dirichlet boundary conditions and the Neumann
boundary conditions. We prove, in this paper, the convergence of the inverse weighted
adjacency matrix ($K^{-1}$) for both boundary conditions. The major difference of this paper
from previous work also lies in the fact that we are working on isoradial graphs, and
the analysis on isoradial graphs is more complicated \cite{CS11}.  The $K^{-1}$ is closely
related to the local statistics of dimers \cite{RK97}, and moreover, the spin-spin correlation of
the Ising model \cite{HS10,JD11}.

\section{Background}

\subsection{Isoradial graphs}\label{ig}

In this subsection we review the definition of the isoradial graph as well as its basic properties. An isoradial graph is a graph which can be embedded into the plane such that each bounded face is inscribable into a circle of common radius. The class of isoradial graphs includes the common regular graphs like the square grid, and hexagonal lattice, but is more general than that. See  Figure \ref{fig:regiso}, Figure \ref{isorum} for examples of isoradial graphs.

\begin{figure}[htb]
\begin{center}
\includegraphics{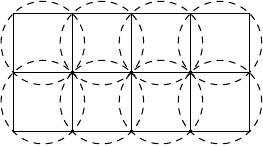}\qquad\qquad\includegraphics{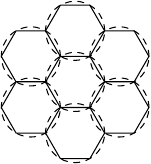}
\caption{Isoradial regular graphs}\label{fig:regiso}
\end{center}
\end{figure}

\begin{figure}[htb]
\begin{center}
\includegraphics{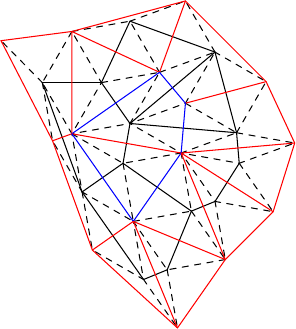}
\caption{Isoradial graph, dual graph and rhombic lattice}\label{isorum}
\end{center}
\end{figure}

An equivalent definition of the isoradial graph is the so-called rhombic tiling. Namely, we can always construct a planar graph $\mathcal{G}_R$ from a planar graph $\mathcal{G}$, such that each face of $\mathcal{G}_R$ is of degree 4. The vertices of $\mathcal{G}_R$ are either vertices of $\mathcal{G}$, or faces of $\mathcal{G}$. Two vertices $v_1,v_2$ of $\mathcal{G}_R$ are connected if and only if $v_1$ is a vertex of $\mathcal{G}$ surrounding the face $v_2$, or vice versa. Hence to see if a graph $\mathcal{G}$ has an isoradial embedding, it suffices to see that if the constructed graph $\mathcal{G}_R$, in which each face is of degree 4, can be embedded onto the plane such that all edges have the same length, i.e., forms a rhombic lattice.  

In a planar graph $\mathcal{G}_R$ with faces of degree 4, a train track is a path of faces (each face being adjacent along an edge to the previous face) which does not turn: on entering a face it exits across the opposite edge. It is proved in \cite{RS05} that a planar graph $\mathcal{G}_R$ with faces of degree 4 has a rhombic embedding if and only if no train track path crosses itself or is periodic; two distinct train tracks cross each other at most once.

In the above construction, each edge of $\mathcal{G}$ corresponds to a rhombus in $\mathcal{G}_R$. A rhombus half angle $\theta_e$ is associated to each edge $e$ of the isoradial graph; it is the angle in $[0,\frac{\pi}{2}]$ formed by the edge and an edge of the corresponding rhombus of $\mathcal{G}_R$, see Figure \ref{isorum}. In Figure \ref{isorum}, the black edges are edges of the primal graph, the blue or red edges are edges of the dual graph, and dashed edges are edges of the rhombic lattice. From the picture it is quite clear that if the primal graph is isoradial and the circumcenter is inside each face of the primal graph, the dual graph is also isoradial.

Each edge of $\mathcal{G}$ is the diagonal of a rhombus in $\mathcal{G}_R$; this diagonal divides the rhombus into two triangles, or half-rhombi. Similarly each dual edge also divides the corresponding rhombus into two half-rhombi. When considering a finite subgraph of the infinite isoradial graph, e.g. the subgraph of the dual isoradial graph bounded by the outer red edges in Figure \ref{isorum}, each interior edge divides a rhombus in the subgraph into two half-rhombi; and each boundary edge corresponds to a triangle, or a half-rhombus of the subgraph; because the other half-rhombus is outside the subgraph.

\subsection{Dimer model}

A perfect matching, or a dimer configuration on an isoradial graph is a choice subset of edges such that each vertex is incident to exactly one edge. 

Let $\mathcal{G}=(V,E)$ be a finite isoradial graph which admits a perfect matching. As discussed in Sect.~\ref{ig}, each edge $e\in E$ is the diagonal of a rhombus, with a unique associated rhombus half-angle $\theta_e\in(0,\frac{\pi}{2})$. We assign to each edge $e\in E$ a critical weight $2\sin\theta_e$. 

Let $\Omega$ be the set of all dimer configurations on $\mathcal{G}$. The probability of a perfect matching $M\in \Omega$, is defined to be
\begin{eqnarray}
\mathbb{P}(M)=\frac{\prod_{e\in M} [2\sin \theta_e]}{Z}.\label{pm}
\end{eqnarray}
Here $Z$ is normalising constant defined by
\begin{eqnarray*}
Z=\sum_{M\in \Omega}\prod_{e\in M}[2\sin\theta_e].
\end{eqnarray*}

\subsection{Height function}\label{dfht}

 The height function is a random, real-valued function defined on faces of the isoradial graph, such that for each perfect matching configuration, the height difference between any two faces is uniquely determined. In other words, each dimer configuration on the isoradial graph gives a unique height function on faces, up to an additive constant.

For any $e\in\mathcal{G}$, let $\mathbb{P}(e)$ be the probability that the edge $e$ is included in a dimer configuration. Assume $w$ is the white endpoint of $e$ and $b$ is the black endpoint of $e$. We define a base flow $\omega_0$ on oriented edges of $\mathcal{G}_{\Omega,\delta}^D$ as follows:
\begin{eqnarray*}
\omega_0(wb)&=&\mathbb{P}(e)\\
\omega_0(bw)&=&-\mathbb{P}(e).
\end{eqnarray*}
where $wb$ means that the edge is oriented from the white vertex to the black vertex, and similarly for $bw$. Since $\mathcal{G}$ admits a dimer configuration, the base flow $\omega_0$ has divergence 1 (resp. $-1$) at each white (resp. black) vertex.

Now, let $M$ be a random perfect matching of $\mathcal{G}$. Then $M$ defines a white-to-black unit flow $\omega_M$ on the edges of $\mathcal{G}$. Namely, $\omega_M$ has value 1 (resp. $-1$) on each edge occupied by a dimer, when the edge is oriented from its white (resp. black) vertex to its black (resp. white) vertex. Other edges have flow 0.

The height function $h$ of a perfect matching $M$, is a real-valued function defined on the faces of $\mathcal{G}$. We define $h$ with respect to the divergence-free flow $\omega_M-\omega_0$ as follows. Let $f_0$ be a fixed face. Define $h(f_0)=0$. Let $f$ be another face of $\mathcal{G}$, and $\gamma_{f_0f}$ be a dual path (a path consisting of dual edges) connecting $f_0$ and $f$. Let $E_{\gamma_{f_0f}}$ be the set of all edges of $\mathcal{G}$ crossed by $\gamma_{f_0f}$. Let $E_{\gamma_{f_0f}}^+$ (resp. $E_{\gamma_{f_0f}}^{-}$) be the set of all edges in $E_{\gamma_{f_0f}}$ such that the white (resp. black) endpoint is on the left of the path $\gamma_{f_0f}$, when travelling along $\gamma_{f_0f}$ from $f_0$ to $f$. Set
\begin{eqnarray}
h(f)&=&\sum_{wb\in E_{\gamma_{f_0f}}^+}(\omega_M-\omega_0)(wb)-\sum_{wb\in E^{-}_{
\gamma_{f_0f}}}(\omega_M-\omega_0)(wb)\label{dh1}\\
&=&\sum_{e\in E^+_{\gamma_{f_0f}}}({\mathbb{I}(e)-\mathbb{P}(e)})-\sum_{e\in E_{\gamma_{f_0f}}^-}(\mathbb{I}(e)-\mathbb{P}(e)),\label{hetd}
\end{eqnarray}
where $\mathbb{I}(e)$ is the indicator of the event that $e$ is present in the dimer configuration. Since $\omega_M-\omega_0$ is a divergence-free flow, and $\mathcal{G}$ is a planar graph, therefore $h$ is well defined, and independent of the path connecting $f_0$ and $f$.

Note that in this definition of height function, we always have the expectation $\mathbb{E}h(f)=0$, for any face $f$ of $\mathcal{G}$.

\subsection{Harmonic function}

The discrete Laplacian operator $\Delta$ on an isoradial graph $\mathcal{G}=(V(\mathcal{G}),E(\mathcal{G}))$, maps a function $H$ defined on $V(\mathcal{G})$, to another function defined on $V(\mathcal{G})$ as follows:
\begin{eqnarray*}
[\Delta H](v)=\left[\sum_{v':v'\sim v}\tan\theta_{vv'}\right]H(v)-\sum_{v':v'\sim v}\left[\tan\theta_{vv'}H(v')\right]
\end{eqnarray*}
where $\theta_{vv'}$ is the rhombus half angle corresponding to the edge $vv'$. 

We can also associate a random walk, or a Markov chain to an isoradial graph, such that the transition probability
\begin{eqnarray*}
p^{v}_1(v_1)=\left\{\begin{array}{cc}\frac{\tan\theta_{vv_1}}{\sum_{v':v'\sim v}\tan\theta_{vv'}}&\mathrm{if}\ v\sim v_1\\ 0& \mathrm{otherwise}\end{array}\right.
\end{eqnarray*}
where $p^v_1(v_1)$ is the probability that a random walk started at $v$ visits $v_1$ at the first step. 

A discrete harmonic function $H$ on an isoradial graph $\mathcal{G}$ is a function defined on vertices of the graph, namely, $H: V(\mathcal{G})\rightarrow \mathbb{R}$ , satisfying
\begin{eqnarray}
[\Delta H](v)=0,\qquad v\in V(\mathcal{G}).\label{mvh}
\end{eqnarray} 
 (\ref{mvh}) can also be considered as a discrete analog of the mean value property.

The mean value property (\ref{mvh}) obviously implies the maximal principle for discrete harmonic functions, i.e., if $\mathcal{G}_{s}$ is a subgraph of the isoradial graph $\mathcal{G}$, and we define boundary vertices of $\mathcal{G}_s$ to be vertices of $\mathcal{G}_s$ that are incident to vertices outside $\mathcal{G}_s$, and interior vertices of $\mathcal{G}_s$ to be vertices in $\mathcal{G}_s$ that are incident only to vertices of $\mathcal{G}_s$. We use $\partial\mathcal{G}_s$ to denote the set of all boundary vertices of $\mathcal{G}_s$. If $H$ is harmonic at any interior vertex of $\mathcal{G}_s$, then the maximal or minimal value of $H$ can only be achieved at boundary vertices of $\mathcal{G}_s$, if $H$ is not constant on all vertices of $\mathcal{G}_s$.

The harmonic extension also has a discrete analog. If $f$ is a real-valued function defined on boundary vertices of $\mathcal{G}_s$, there exists a unique discrete harmonic function $H$ on $V(\mathcal{G}_s)$ such that the values of $H$ on boundary vertices of $\mathcal{G}_s$ are the same as those of $f$. In fact, $H$ can be written down explicitly as follows
\begin{eqnarray}
H(u)=\sum_{a\in\partial \mathcal{G}_s}\omega(u;{a};V(\mathcal{G}_s))\cdot f(a);\label{het}
\end{eqnarray}
for all $u\in V(\mathcal{G}_s)$, where $\omega(u;{a};V(\mathcal{G}_s))$ is the probability that a random walk started at $u$ first hit the boundary $\partial \mathcal{G}_s$ at $a$. $\omega(u; \cdot;V(\mathcal{G}_s))$ is a harmonic function in $u$ and a probability measure on $\partial{G}_s$. If $E\subseteq \partial{G}_s$, 
$\omega(u; E;V(\mathcal{G}_s))$ is the exit probability through $E$ of the random walk started at $u$.

\subsection{Gaussian free field}

The 2-dimensional Gaussian free field (GFF), or massless free field, is a natural 2-dimensional time analog of the Brownian motion.

Let $\Omega$ be a simply-connected bounded domain of the complex plane $\mathbb{C}$. Let $C_0^{\infty}(\Omega)$ be the space of smooth, real-valued functions that are supported on compact subsets of $\Omega$. The GFF on $\Omega$, can also be considered as a Gaussian random vector on the infinite-dimensional Hilbert space $H_0^1(\Omega)$, where $H_0^1(\Omega)$ is the completion of $C_0^{\infty}(\Omega)$ under the $L^2$ norm of derivatives, see \cite{SS07}.

Let $\{e_i\}_{i\geq 1}$ be an $L^2$-orthornormal eigenfunctions for the Laplacian $\Delta=\frac{\partial^2}{\partial x^2}+\frac{\partial^2}{\partial y^2}$ with Dirichlet boundary conditions (i.e., $e_i=0$ on $\partial \Omega$). Let $\eta_i$ be the eigenvalue of $e_i$. The GFF on $\Omega$, $\mathcal{F}$, is a random distribution (continuous linear functional) on $C^1$ functions of $\Omega$, such that, for any $C^1$ function $\phi$,
\begin{eqnarray*}
\mathcal{F}(\phi)=\sum_{i\geq 1}\frac{\alpha_i}{(-\eta_i)^{\frac{1}{2}}}\int_{\Omega}\phi(x,y) e_i(x,y) dx dy,
\end{eqnarray*}
where $\{\alpha_i\}_{i\geq 1}$ are i.i.d. Gaussian random variables with mean 0 and variance 1.

\section{Convergence of discrete holomorphic observables}\label{cvdo}

In this section we introduce the discrete holomorphic observable for perfect matchings on isoradial graphs, namely, the so-called inverse weighted adjacency matrix entries, and prove its convergence in the scaling limit under special boundary conditions.

\subsection{Discrete approximation}\label{da}

In this subsection we discuss the basic setting and assumptions for the discrete approximation, under which the convergence of discrete holomorphic observables will be proved. 

Let $\mathcal{G}_{\delta}$ be an isoradial graph on the whole plane, i.e., each face is inscribable in a circle of radius $\delta$.
Assume all the rhombus half-angles of $\mathcal{G}_{\delta}$ are bounded uniformly away from $0$ and $\frac{\pi}{2}$, i.e., assume there exists a constant $c_0>0$, such that for all edges $e$, the rhombus half angle $\theta_e\in\left[c_0,\frac{\pi}{2}-c_0\right]$.
  Let $\Omega$ be a simply-connected bounded domain of the complex plane $\mathbb{C}$ bounded by a simple closed curve. Let $\mathcal{G}_{\Omega,\delta}$ be an isoradial subgraph of $\mathcal{G}_{
\delta}$, consisting of faces, i.e., each face of $\mathcal{G}_{\Omega,\delta}$ is inscribable in a circle of radius $\delta$. Assume $\mathcal{G}_{\Omega,\delta}$ is also simply-connected, i.e., the boundary of $\mathcal{G}_{\Omega,\delta}$, $\partial\mathcal{G}_{\Omega,\delta}$ has only one connected component. Let $\mathcal{G}'_{\Omega,\delta}$ be the interior dual graph of $\mathcal{G}_{\Omega,\delta}$, and assume that $\mathcal{G}_{\Omega,\delta}$ is the interior dual graph of $\mathcal{G}''_{\Omega,\delta}$, as in Figure \ref{isorum}, where the primal graph $\mathcal{G}_{\Omega,\delta}$ is bounded by the outer boundary consisting of black edges, $\mathcal{G}'_{\Omega,\delta}$, the interior dual graph of $\mathcal{G}_{\Omega,\delta}$, is the graph bounded by the outer boundary consisting of blue edges; and $\mathcal{G}_{\Omega,\delta}''$, which has $\mathcal{G}_{\Omega,\delta}$ as its interior dual graph, is the graph bounded by the outer boundary consisting of red edges.

We make the following assumption.

\begin{assumption}\label{asp}Assume as $\delta\rightarrow 0$, $\mathcal{G}_{\Omega,\delta}$ approximates $\Omega$ in the following sense:
\item $\mathcal{G}_{\Omega,\delta}$ is the largest subgraph of $\mathcal{G}_{\delta}$, consisting of faces and satisfying $\mathcal{G}_{\Omega,\delta}\subset\Omega$.
\end{assumption}

Given $\mathcal{G}_{\Omega,\delta}$, we construct another graph $\mathcal{G}_{\Omega,\delta}^D$ as follows. First of all, let $\tilde{\mathcal{G}}_{\Omega,\delta}^D$ be a graph with a vertex for each vertex, edge and face of $\mathcal{G}_{\Omega,\delta}$, and an edge for each half-edge of $\mathcal{G}_{\Omega,\delta}$ and each half dual-edge as well. The graph $\tilde{\mathcal{G}}_{\Omega,\delta}^D$ is a bipartite planar graph, with black vertices of two types: vertices of $\mathcal{G}_{\Omega,\delta}$ and faces of $\mathcal{G}_{\Omega,\delta}$. White vertices of $\tilde{\mathcal{G}}_{\Omega,\delta}^D$ corresponds to edges of $\mathcal{G}_{\Omega,\delta}$, and dual edges in $\mathcal{G}'_{\Omega,\delta}$. According to the Euler's theorem, $\tilde{\mathcal{G}}_{\Omega,\delta}^D$ has one more black vertices than white vertices. Let $\mathcal{G}_{\Omega,\delta}^D$ be obtained from $\tilde{\mathcal{G}}_{\Omega,\delta}^D$ by removing a black vertex $b_{\delta,0}$ on the boundary. Assume there exists $z_0\in\partial\Omega$, such that $\lim_{\delta\rightarrow 0}b_{\delta,0}=z_0$.

We claim that the graph $\mathcal{G}_{\Omega,\delta}^D$, as constructed above, admits a perfect matching. To see that, we will construct a perfect matching for the graph $\mathcal{G}_{\Omega,\delta}^D$, and the construction is to use the Temperley's bijection \cite{Te,BP,KPW} between perfect matchings and spanning trees, as discussed in the square grid case in \cite{RK00}.  

First of all, the graph $\mathcal{G}_{\Omega,\delta}$ has a spanning tree rooted at $b_{\delta,0}$, i.e., we can give an orientation to each edge of the spanning tree, so that each vertex, except $b_{0,\delta}$, has a unique incoming edge; and $b_{0,\delta}$ has only outgoing edges. For each vertex $v$ of $\mathcal{G}_{\Omega,\delta}$ except $b_{\delta,0}$, we pair the vertex with the white vertex of $\mathcal{G}_{\Omega,\delta}^D$, $w$, on the unique incoming edge of the vertex $v$. In other words, we assume the edge $vw$ of the graph $\mathcal{G}_{\Omega,\delta}^D$ is included in a perfect matching of $\mathcal{G}_{\Omega,\delta}^D$.

Let $W_1$ be the set of all white vertices of $\mathcal{G}_{\Omega,\delta}^D$, included in a perfect matching edge with a black vertex in $\mathcal{G}_{\Omega,\delta}$, and $W_2$ be the set of all the other white vertices of $\mathcal{G}_{\Omega,\delta}^D$. Note that each white vertex of $\mathcal{G}_{\Omega,\delta}^D$ is the intersection of an edge of $\mathcal{G}_{\delta}$ and a dual edge $\mathcal{G}_{\delta}'$. The dual edges passing vertices in $W_2$ pass all the vertices of $\mathcal{G}_{\Omega,\delta}'$, and contain no cycles, because the spanning tree of $\mathcal{G}_{\Omega,\delta}$ is connected. In other words, each connected component of the dual edges passing vertices in $W_2$ is a tree, with a unique vertex in $\mathcal{G}_{\Omega,\delta}''\setminus \mathcal{G}_{\Omega,\delta}'$. Let $T$ be such a tree and $v'$ be the unique vertex of $T$ in $\mathcal{G}_{\Omega,\delta}''\setminus\mathcal{G}_{\Omega,\delta}'$. Give each edge of $T$ an orientation, so that $v'$ has only outgoing edges and all the other vertices of $T$ has exactly one incoming edge. We pair each vertex of $T$, except $v'$, with the white vertex of $\mathcal{G}_{\Omega,\delta}^D$ on the unique incoming edge of that vertex. Then we obtain a perfect matching of $\mathcal{G}_{\Omega,\delta}^D$.

An isoradial embedding of $\mathcal{G}_{\Omega,\delta}$ gives rise to an isoradial embedding of $\tilde{\mathcal{G}}_{\Omega,\delta}^D$. Each rhombus corresponding to an interior edge of $\mathcal{G}_{\Omega,\delta}$ with half-angle $\theta$ is divided into four congruent rhombi in $\tilde{\mathcal{G}}_{\Omega,\delta}^D$. Each triangle (half-rhombus) corresponding to a boundary edge of $\mathcal{G}_{\Omega,\delta}$ is divided into two congruent triangles and one rhombus. The two triangles are part of a rhombus with half-angle $\theta$, and the rhombus has half-angle $\frac{\pi}{2}-\theta$. $\mathcal{G}_{\Omega,\delta}^D$, as a subgraph of $\tilde{\mathcal{G}}_{\Omega,\delta}^D$, inherited an isoradial embedding of $\tilde{\mathcal{G}}_{\Omega,\delta}^D$, in which each face is inscribable in a circle of radius $\frac{\delta}{2}$. 

We define a symmetric matrix $\overline{\partial}$, indexed by vertices of $\mathcal{G}_{\Omega,\delta}^D$, as follows: if $v_1$ and $v_2$ are not adjacent, $\overline{\partial}(v_1,v_2)=0$. If $w$ and $b$ are adjacent vertices, $w$ being white and $b$ being black, then $\overline{\partial}(w,b)=\overline{\partial}(b,w)$ is the complex number of length $\nu(w,b)=2\sin\theta$, with direction pointing from $w$ to $b$. $2\sin\theta$'s are called the critical edge weights for isoradial graphs, which were introduced in \cite{RK02}, as well as the following gauge transformation.

Let $\overline{D}$ be obtained from $\overline{\partial}$ by multiplying edge weights of $\mathcal{G}_{\Omega,\delta}^D$ around each white vertex (coming from an edge of $\mathcal{G}_{\Omega,\delta}$ of half-angle $\theta$) by $\frac{1}{2\sqrt{\sin\theta\cos\theta}}$, that is $\overline{D}(w,b)=S\overline{\partial}(w,b)S^*$, where $S$ is the diagonal matrix defined by $S(w,w)=\frac{1}{2\sqrt{\sin\theta(w)\cos\theta(w)}}$, and $S(b,b)=1$. Then an edge of $\mathcal{G}_{\Omega,\delta}^D$ coming from a ``primal'' edge of $\mathcal{G}_{\Omega,\delta}$ has weight $\sqrt{\tan\theta}$ for $\overline{D}$, and an edge coming from a dual edge has weight $\frac{1}{\sqrt{\tan\theta}}$ for $\overline{D}$. (These edges have weight $2\sin\theta$, $2\cos\theta$ respectively for $\overline{\partial}$.) Let $\overline{D}^*$ denote the transpose conjugate of $\overline{D}$.

\begin{lemma}\label{dn}Restricted to vertices of $\mathcal{G}_{\Omega,\delta}$, we have $\overline{D}^*\overline{D}=\Delta_n$, where $\Delta_n$ is the Laplacian operator of $\mathcal{G}_{\Omega,\delta}$ with $0$ Neumann boundary conditions, and the boundary value at the removed vertex $b_{\delta,0}$ is $0$. Restricted to faces of $\mathcal{G}_{\Omega,\delta}$, we have $\overline{D}^*\overline{D}=\Delta_d$, where $\Delta_d$ is the Laplacian operator of $\mathcal{G}_{\Omega,\delta}'$, the interior dual graph of $\mathcal{G}_{\Omega,\delta}$ with $0$ Dirichlet boundary conditions.
\end{lemma}
\begin{proof} Using the same computation as in \cite{RK02}, it is straightforward to check the following
\begin{enumerate}

\item Let $b$ be a black vertex of $\mathcal{G}_{\Omega,\delta}^D$. Assume in $\mathcal{G}_{\Omega,\delta}^D$, $b$ has neighbors $w_1,...,w_k$, so that for $1\leq j\leq k$, the edge $bw_j$ has rhombus half-angle $\theta_j$, we have
\begin{eqnarray*}
\overline{D}^*\overline{D}(b,b)=\sum_{j=1}^{k}\tan\theta_j.
\end{eqnarray*}

\item If $b,b'$ are adjacent vertices in $\mathcal{G}_{\Omega,\delta}$, so that the edge $bb'$ has rhombus half-angle $\theta$ in $\mathcal{G}_{\Omega,\delta}$, we have
\begin{eqnarray*}
\overline{D}^*\overline{D}(b,b')=-\tan\theta,
\end{eqnarray*}
and similar computations apply for the case when $b,b'$ are adjacent vertices in $\mathcal{G}'_{\Omega,\delta}$.

\item If $b,b'$ are of distance 2 in $\mathcal{G}_{\Omega,\delta}^D$, but correspond to a vertex and a face of $\mathcal{G}_{\Omega,\delta}$, then
\begin{eqnarray*}
\overline{D}^*\overline{D}(b,b')=0.
\end{eqnarray*}

\end{enumerate}

Unlike the computations in \cite{RK02}, which consider the whole-plane operators with no boundaries, we also consider boundary conditions here.
Let $f$ be a function defined on the set of vertices of $\mathcal{G}_{\Omega,\delta}^D$. First assume that $b$ is not adjacent to $b_{\delta,0}$ in $\mathcal{G}_{\Omega,\delta}$. If $b$ is an interior vertex of $\mathcal{G}_{\Omega,\delta}$, then the above computation gives
\begin{eqnarray*}
\overline{D}^*\overline{D}f(b)=\sum_{b_k\sim b}\tan\theta_k(f(b)-f(b_k)).
\end{eqnarray*}
This is the same as the effect of discrete Laplacian operator on the whole-plane graph $\mathcal{G}_{\delta}$. If $b$ is a boundary vertex of $\mathcal{G}_{\Omega,\delta}$, then the above computation gives
\begin{eqnarray*}
\overline{D}^*\overline{D}f(b)=\sum_{b_k\sim b,b_k\in\mathcal{G}_{\Omega,\delta}}\tan\theta_k(f(b)-f(b_k)).
\end{eqnarray*}
This is the same as the effect of discrete Laplacian operator on the whole-plane graph $\mathcal{G}_{\delta}$, if we require that on all the vertices adjacent to $b$ in $\mathcal{G}_{\delta}\setminus\mathcal{G}_{\Omega,\delta}$, $f$ has the same value as $f(b)$. If we consider the boundary of the domain consisting of dual edges of the adjacent edges of $b$ in $\mathcal{G}_{\delta}\setminus\mathcal{G}_{\Omega,\delta}$, i.e., $\partial\mathcal{G}_{\Omega,\delta}''$, this is the discrete analogue of $0$ normal derivative along the boundary.

Now assume $b$ is an adjacent vertex of $b_{\delta,0}$ in $\mathcal{G}_{\Omega,\delta}$. After imposing the Neumann boundary conditions on vertices adjacent to $b$, but not in $\tilde{G}_{\Omega,\delta}$, we have
\begin{eqnarray*}
\overline{D}^*\overline{D}f(b)=\sum_{b_k\sim b,b_k\neq b_{\delta,0}}\tan\theta_k(f(b)-f(b_k))+\tan\theta_0f(b).
\end{eqnarray*}
This is the same as the effect of the discrete Laplacian on the whole-plane graph $\mathcal{G}_{\delta}$, by requiring that $f(b_{\delta,0})=0$.

Now assume $b'$ is a vertex of $\mathcal{G}_{\Omega,\delta}'$, and $h$ be a function defined on vertices of $\mathcal{G}_{\Omega,\delta}'$. We have
\begin{eqnarray*}
\overline{D}^*\overline{D}h(b')=\sum_{b_k'\sim b',b_k'\in \mathcal{G}_{\Omega,\delta}'}\tan\theta_k(h(b')-h(b_k'))+\sum_{b_j'\sim b', b_j'\in\mathcal{G}_{\Omega,\delta}''\setminus\mathcal{G}_{\Omega,\delta}'}\tan\theta_j h(b').
\end{eqnarray*}

Similar argument shows that the effect of $\overline{D}^*\overline{D}$ for $h$ on each vertex of $\mathcal{G}_{\Omega,\delta}'$ will be the same as the discrete Laplacian of the whole plane graph $\mathcal{G}_{\delta}'$ for $h$ at the vertex of $\mathcal{G}_{\delta}'$, if we require the value of $h$ on $\partial \mathcal{G}_{\Omega,\delta}''$ to be $0$.
\end{proof}

The following lemma states the relation between the inverse matrix $\overline{\partial}^{-1}$ with the local statistics of the perfect matchings. It was first discovered for the hexagonal lattice in \cite{RK97}, then reformulated in the setting of isoradial graphs in \cite{RK02}. It also appears in \cite{dd}, although in a less general form.

\begin{lemma}(\cite{RK97,RK02})The dimer partition function on $\mathcal{G}_{\Omega,\delta}^D$ is equal to $\sqrt{|\det\overline{\partial}|}$. The probability of edges $w_1b_1,...,w_kb_k$ occurring in a configuration chosen with respect to Boltzmann measure $\mu$ is 
\begin{eqnarray*}
\prod_{i=1}^{k}\overline{\partial}(w_i,b_i)\det_{1\leq i,j\leq k}\overline{\partial}^{-1}(w_i,b_j)
\end{eqnarray*}
\end{lemma}

\subsection{Convergence with Dirichlet boundary conditions}

From Lemma \ref{dn}, we know that there are two types of $\overline{D}^{-1}$ entries, depending on whether the black vertex is a vertex of the primal graph or a vertex of the dual graph. In this subsection we deal with the convergence when the black vertex is a vertex of the dual graph. We start with a technical lemma proved in \cite{CS11}, which says that uniformly bounded sequence of harmonic functions, has a uniformly convergent subsequence on compact subsets of the domain.

\begin{lemma}\label{dhch}(\cite{CS11})\label{cvhd}Let $H^{\delta}:V(\mathcal{G}_{\Omega,\delta}) \rightarrow \mathbb{R}$ be a (real-valued) discrete harmonic function defined on vertices of the isoradial graph $\mathcal{G}_{\Omega,\delta}$ with $\delta\rightarrow 0$. If $\{H^{\delta}\}_{\delta}$ are uniformly bounded  on $\Omega$, i.e.,
\begin{eqnarray*}
\max_{u\in V(\mathcal{G}_{\Omega,\delta})}|H^{\delta}(u)|\leq M<+\infty,
\end{eqnarray*}
where $M$ is a constant independent of $\delta$, then there exists a subsequence $\delta_{k}\rightarrow 0$, and two functions $h:\Omega\rightarrow \mathbb{R}$, $f:\Omega\rightarrow\mathbb{C}$, such that
\begin{eqnarray*}
H^{\delta^k}\rightarrow h,\qquad uniformly\ on\ compact\ subsets\ K\subset\Omega,
\end{eqnarray*}
and
\begin{eqnarray*}
\frac{H^{\delta_k}(u_k^{+})-H^{\delta_k}(u_k^{-})}{|u_k^{+}-u_k^{-}|}\rightarrow\Re\left[f(u)\frac{u_k^{+}-u_k^{-}}{|u_k^{+}-u_k^{-}|}\right],
\end{eqnarray*}
if $u_k^{\pm}\in V(\mathcal{G}_{\Omega,\delta})$, $u_k^{+}\sim u_k^{-}$ and $u_k^{\pm}\rightarrow u\in K\subset \Omega$ as $k\rightarrow\infty$. Moreover, the limit function $h$, $|h|\leq M$, is harmonic in $\Omega$ and $f=h_x'-ih_y'=2\partial h$ is analytic in $\Omega$.
\end{lemma}

Assume $v_1$ is a white vertex of $\mathcal{G}_{\Omega,\delta}^D$, consider $\overline{D}^{-1}$ as a function of $v_2$, where $v_2$ is a black vertex of $\mathcal{G}_{\Omega,\delta}^D$. $v_1$ is adjacent to four vertices: two vertices corresponding to vertices of $\mathcal{G}_{\delta}$, denoted by $b_1$ and $b_2$, and two vertices corresponding to faces of $\mathcal{G}_{\delta}$, denoted by $b_3$ and $b_4$. Let $\xi_1,\xi_2,\xi_3,\xi_4$ denote the unit vector pointing from $v_1$ to $b_1,b_2,b_3,b_4$, respectively. Let $\theta_{v_1}$ denote the rhombus half-angle corresponding to the edge $b_1b_2$. Explicit computations show that
\begin{eqnarray*}
\overline{D}^*\overline{D}\overline{D}^{-1}(v_2,v_1)&=&\overline{D}^*(v_2,v_1)\\
&=&\left\{\begin{array}{cc}0&\mathrm{if}\ v_1\ \mathrm{and}\ v_2\ \mathrm{are\ not\ adjacent}\\ \frac{\sqrt{\tan\theta_{v_1}}}{\xi_i}&\mathrm{if}\ v_2=b_i,i=1,2\\ \frac{1}{\sqrt{\tan\theta_{v_1}}\xi_i}&\mathrm{if}\ v_2=b_i, i=3,4 \end{array}\right.
\end{eqnarray*}

\begin{lemma}\label{db}Let $z_1$ be a point in the interior of $\Omega$, and let $z_2\in\overline{\Omega}$, $z_2\neq z_1$. Let $v_1$ be a white vertex of $\mathcal{G}_{\Omega,\delta}^D$ nearest to $z_1$, which corresponds to an edge in $\mathcal{G}_{\Omega,\delta}$ with a fixed rhombus half-angle $\theta_{v_1}$, and let $v_2$ be a vertex of $\mathcal{G}_{\Omega,\delta}'$ nearest to $z_2$. Let $\xi_{3,\delta}$ be the unit vector pointing from $b_4$ to $b_3$ in $\mathcal{G}'_{\Omega,\delta}$, as discussed above. Furthermore, assume that there exists a fixed unit vector $\xi_3$, so that
\begin{eqnarray*}
\lim_{\delta\rightarrow 0}\xi_{3,\delta}=\xi_3.
\end{eqnarray*}
Then
\begin{eqnarray*}
\lim_{\delta\rightarrow 0}\frac{1}{\delta}\overline{D}^{-1}(v_2,v_1)&=&2\sqrt{\sin\theta_{v_1}\cos\theta_{v_1}}\frac{\partial_1}{\partial\xi_3}g_{\Omega}(z_1,z_2)\\
\lim_{\delta\rightarrow 0}\frac{1}{\delta}\overline{\partial}^{-1}_{\Omega,\delta}(v_2,v_1)&=&\frac{\partial_1}{\partial\xi_3}g_{\Omega}(z_1,z_2)
\end{eqnarray*}
where $g_{\Omega}(z_1,z_2)$ is the Green's function of the region $\Omega$, and the derivative is the directional derivative along the $\xi_3$ with respect to the first variable. Moreover, let $\mathcal{D}$ be the diagonal set of $\overline{\Omega}\times\overline{\Omega}$, then the convergence is uniform for $(z_1,z_2)$ in compact subsets of $\Omega\times\overline{\Omega}\setminus\mathcal{D}$.
\end{lemma}

\begin{proof}Let $\overline{D}_0$ be the corresponding matrix defined for the whole-plane graph $\mathcal{G}_{\delta}^D$, obtained by the superposition of $\mathcal{G}_{\delta}$ and $\mathcal{G}'_{\delta}$. Then
\begin{eqnarray*}
\frac{1}{\delta}\overline{D}_0^{-1}(v_2,v_1)=\frac{1}{\delta\xi_3\sqrt{\tan\theta_{v_1}}}[G'_{\delta}(v_2,b_3)-G'_{\delta}(v_2,b_4)]\qquad \mathrm{if}\ v_2\in V(\mathcal{G}'_{\Omega,\delta}),
\end{eqnarray*}
where $G_{\delta}'$ is the discrete Green's function on $\mathcal{G}_{\delta}'$, the dual graph of $\mathcal{G}_{\delta}$ on the whole plane. It is proved in \cite{RK02,Bu} that
\begin{eqnarray*}
G_{\delta}'(v_1,v_2)=-\frac{1}{2\pi}\log\left|\frac{v_2-v_1}{\delta}\right|-\frac{\gamma_{Euler}+\log 2}{2\pi}+O\left(\frac{\delta^2}{|v_2-v_1|^2}\right)
\end{eqnarray*}
where $\gamma_{Euler}$ is the Euler's constant. Then
\begin{eqnarray}
\frac{1}{\delta\xi_3\sqrt{\tan\theta_{v_1}}}[G'_{\delta}(v_2,b_3)-G'_{\delta}(v_2,b_4)]=\frac{\sqrt{\sin\theta_{v_1}\cos\theta_{v_1}}}{\pi\xi_3}\Re\frac{\xi_3}{v_2-v_1}+O(\delta),\label{lmd}
\end{eqnarray}
and the coefficient of $\delta$ in the error term is bounded uniformly for any $(v_1,v_2)$ in compact subsets of $\Omega\times\overline{\Omega}\setminus \mathcal{D}$. Fix $v_1$, and let
\begin{eqnarray*}
H_{\delta}(v_2)=\frac{1}{\delta}[\overline{D}^{-1}(v_2,v_1)-\overline{D}_0^{-1}(v_2,v_1)],
\end{eqnarray*}

Evidently, $H_{\delta}(v_2)$ is satisfies the following conditions
\begin{eqnarray*}
&&\Delta H_{\delta}(v)=0\qquad\qquad\forall v\in V(\mathcal{G}'_{\Omega,\delta})\\
&&H_{\delta}(v)|_{v\in V(\partial \mathcal{G}''_{\Omega,\delta})}=-\frac{\sqrt{\sin\theta_{v_1}}\cos\theta_{v_1}}{\pi\xi_3}\Re\frac{\xi_3}{v-v_1}+O(\delta)+o(1).
\end{eqnarray*}

Using the same argument as in the proof of Theorem 3.10 in \cite{CS11}, we can show that as $\delta\rightarrow 0$, the sequence $\{H_{\delta}\}_{\delta}$ converges uniformly on compact subsets of $\Omega$ to a continuous (complex valued) harmonic function $h(z)$ on $\Omega$, with boundary value given by
\begin{eqnarray*}
h(z)|_{z\in\partial \Omega}=-\frac{\sqrt{\sin\theta_{v_1}\cos\theta_{v_1}}}{\pi\xi_3}\Re\frac{\xi_3}{z-z_1}.
\end{eqnarray*}

Then $\overline{D}^{-1}(z_1,z_2)$ converges to the function
\begin{eqnarray*}
h(z_2)+\frac{\sqrt{\sin\theta_{v_1}\cos\theta_{v_1}}}{\pi\xi_3}\Re\frac{\xi_3}{z_2-z_1}. 
\end{eqnarray*}
This is $2\sqrt{\sin\theta_{v_1}\cos\theta_{v_1}}$ times the directional derivative of the continuous Green's function.

The expression for $\overline{\partial}^{-1}_{\Omega,\delta}$ follows from the following gauge equivalence identity
\begin{eqnarray*}
\overline{D}(v_2,v_1)=\frac{\overline{\partial}_{\Omega,\delta}(v_2,v_1)}{2\sqrt{\sin\theta_{v_1}\cos\theta_{v_1}}}.
\end{eqnarray*}

\end{proof}

\subsection{Near the straight boundary}\label{nf}

Assume $\partial\Omega$ has a straight portion, denoted by $L_0$. Assume $\partial\mathcal{G}_{\Omega,\delta}''$ has a straight portion approximating $L_0$, denoted by $L_{0,\delta}$. In this subsection we explore the behaviour of $\overline{D}^{-1}$ when one or both variables are near the straight boundary $L_{0,\delta}$.

We can extend the isoradial graph $\mathcal{G}_{\Omega,\delta}''$ across the flat boundary $L_{0,\delta}$, such that the isoradial faces are symmetric with respect to the axis $L_{0,\delta}$. We can also extend the Green's function $G'_{\Omega,\delta}(v_1,v_2)$ across $L_{0,\delta}$ as follows. If $v_2$ and $v_2^*$ are symmetric vertices with respect to $L_{0,\delta}$, we define
\begin{eqnarray*}
G'_{\Omega,\delta}(v_1,v_2^*)=-G'_{\Omega,\delta}(v_1,v_2).
\end{eqnarray*}

Let $z_1$ be a point in $\overline{\Omega}$, and $z_2\in L_0$ such that $z_2\neq z_1$. Using the extensions above, we have a discrete harmonic function $G_{\Omega,\delta}'(z_1,\cdot)$ in a neighborhood of $z_2$, and they are uniformly bounded in the neighborhood of $z_2$ with a bound independent of $\delta$. By Lemma \ref{cvhd}, both the harmonic function and its directional derivatives converge uniformly in a neighborhood of $z_2$.

\subsection{Convergence with Neumann boundary conditions}\label{nb}

In this subsection, we deal with the entries of $\overline{D}^{-1}$ where the black vertex is a vertex of the primal graph. It has a boundary condition which is a discrete analog of the Neumann boundary condition, given that the difference of observables along the normal direction of each boundary edge is 0.

 The following lemma states the convergence of $\overline{D}^{-1}$ when the black vertex is a vertex of the primal graph.

\begin{lemma}\label{nb}Let $\Omega$ be a simply-connected, bounded domain bounded by a simple closed curve, and $\mathcal{G}_{\Omega,\delta}$ be an isoradial graph with common radius $\delta$, such that $\mathcal{G}_{\Omega,\delta}$ is the largest subgraph of $\mathcal{G}_{\delta}$ (the isoradial graph on the whole plane), consisting of faces of $\mathcal{G}_{\delta}$, lying completely inside $\Omega$. Let $z_1$ be a point in the interior of $\Omega$, and let $z_2\in\overline{\Omega}$ and $z_2\neq z_1$. Let $v_1$ be a white vertex of $\mathcal{G}_{\Omega,\delta}^D$ and $v_2$ be a vertex of $\mathcal{G}_{\Omega,\delta}$. Assume as $\delta\rightarrow 0$, $(v_1,v_2)$ approximates $(z_1,z_2)$ in the following sense
\begin{enumerate}
\item there exists $C>0$, independent of $\delta$, such that $|v_2-z_2|<C\delta$, and $|v_1-z_1|<C\delta$;
\item as $\delta\rightarrow 0$, the direction in $\mathcal{G}'_{\Omega,\delta}$ corresponding to $v_1$ are fixed to be $\xi_3$.
\end{enumerate}
Then fix $v_1$, there exists a subsequence such that the real and imaginary parts of $\frac{1}{\delta}(\overline{\partial}_{\Omega,\delta}^{-1}(v_1,b)-\overline{\partial}_0^{-1}(v_1,b))$ converge uniformly on compact subsets of $\Omega\times \Omega$ to harmonic functions on $b$, as $\delta\rightarrow 0$. Let $F(z_1,z_2)$ be the limit of $\frac{1}{\delta}\overline{\partial}^{-1}_{\Omega,\delta}(v_1,b)$, $b\in V(\mathcal{G}_{\Omega,\delta})$, and let $v_2$ be a sequence of black vertices approximating $z_2$ which have incident edges in $\mathcal{G}_{\Omega,\delta}$ with direction equal to a fixed vector $\alpha$, then
\begin{eqnarray*}
\frac{\partial_2 F(z_1,z_2)}{\partial\alpha}=\frac{\partial_1}{\partial\xi_3}\frac{\partial_2}{\partial(i\alpha)} [g_{\Omega}(z_1,z_2)],
\end{eqnarray*}
where $\frac{\partial_2}{\partial\alpha}$ (resp. $\frac{\partial_2}{\partial (i\alpha)}$ ) is the directional derivative of the second variable with respect to the $\alpha$ (resp. $i\alpha$) direction, for $z_1\neq z_2$, $(z_1,z_2)\in\Omega\times\Omega$.
\end{lemma}
\begin{proof}Recall that there are two types of black vertices $B_0$ and $B_1$, where $B_0$ corresponds to vertices of the graph $\mathcal{G}_{\Omega,\delta}$, and $B_1$ corresponds to vertices of the dual graph $\mathcal{G}_{\Omega,\delta}''$.

Fix a white vertex $v_1$. Define $f\in \mathbb{C}^{B_0}$, as a function on $B_0$ vertices and $g\in \mathbb{C}^{B_1}$, as a function on $B_1$ vertices as follows
\begin{eqnarray*}
f(b)&=&\overline{\partial}_{\Omega,\delta}^{-1}(b,v_1)\qquad\mathrm{if}\ b\in B_0\\
g(b)&=&\overline{\partial}_{\Omega,\delta}^{-1}(b,v_1)\qquad\mathrm{if}\  b\in B_1.
\end{eqnarray*}
Let $w$ be a vertex of $\mathcal{G}_{\Omega,\delta}^D$. Let $b_x,b_y$ be two adjacent vertices of $w$ in $B_0$, and $b_p,b_q$ be two adjacent vertices of $w$ in $B_1$. Then
\begin{eqnarray*}
\delta_w(v_1)=\overline{\partial}_{\Omega,\delta}\overline{\partial}_{\Omega,\delta}^{-1}(w,v_1)=\sum_{i=x,y,p,q}\overline{\partial}_0(w,b_i)\overline{\partial}_{\Omega,\delta}^{-1}(b_i,v_1),
\end{eqnarray*}
where $\overline{\partial}_0$ is the operator on the whole-plane bipartite isoradial graph $\mathcal{G}_{\delta}^D$, namely, the superposition of $\mathcal{G}_{\delta}$ and $\mathcal{G}_{\delta}'$. Note that if $b_i\in\partial\mathcal{G}_{\Omega,\delta}''$, then $\overline{\partial}_{\Omega,\delta}^{-1}(b_i,v_1)=0$. Let $\alpha$, (resp. $-\alpha, i\alpha,-i\alpha)$ be a unit vector from $w$ to $b_x$ (resp. $b_y,b_p,b_q)$, we then have
\begin{eqnarray*}
2\sin\theta_{w}\alpha(f(b_x)-f(b_y))+2\cos\theta_w i\alpha(g(b_p)-g(b_q))=\left\{\begin{array}{cc}0&\mathrm{if}\ w\neq v_1\\1&\mathrm{if}\ w=v_1\end{array}\right.,
\end{eqnarray*}
where $\theta_w$ is the rhombus half-angle corresponding to the edge $b_xb_y$. Hence if $w\neq v_1$, we have
\begin{eqnarray*}
f(b_x)-f(b_y)=-i\frac{\cos\theta_w}{\sin\theta_w}(g(b_p)-g(b_q)).
\end{eqnarray*}

Let $b_{0,\delta}$ be the removed vertex, recall that $f(b_{0,\delta})=0$. Let $b$ be an arbitrary vertex of $\mathcal{G}_{\Omega,\delta}$. We use a path, consisting of $n$ edges of $\mathcal{G}_{\Omega,\delta}$ to connect $b_{0,\delta}$ and $b$, such that $n\delta$ is uniformly bounded for any $\delta$, and there exists a fixed $c_0>0$ such that the path does not intersect the ball $B(v_1,c_0)$ for any $\delta$. Assume the vertices in $\mathcal{G}_{\Omega,\delta}$ visited by the path are
\begin{eqnarray*}
b_0(=b_{0,\delta}),b_1,b_2,...,b_n(=b),
\end{eqnarray*}
in which $b_i$ and $b_{i+1}$ ($0\leq i\leq n-1$) are adjacent in $\mathcal{G}_{\Omega,\delta}$. Then
\begin{eqnarray}
\overline{\partial}^{-1}_{\Omega,\delta}(b,v_1)=f(b)&=&\sum_{k=1}^{n}f(b_k)-f(b_{k-1})\\
&=&-i\sum_{i=1}^{n}\frac{\cos\theta_k}{\sin\theta_k}(g(b_{p,k})-g(b_{q,k}))\label{dre}
\end{eqnarray}
where $\theta_k$ is the rhombus half-angle corresponding to the edge $b_kb_{k-1}$, and $b_kb_{p,k}b_{k-1}b_{q,k}$ is the rhombus (where the vertices are in counter-clockwise order) corresponding to the edge $b_{k-1}b_k$. 

Note that in a compact subset of $\Omega\setminus\{z_1\}$, $g(b)$ converges uniformly to a continuous harmonic function, as is proved in Lemma \ref{db}. Moreover, $g(b)$ is uniformly bounded on a compact subset of $\Omega\setminus\{z_1\}$, by Lemma \ref{dhch}, the directional difference of $g(b)$ converges uniformly on a compact subset of $\Omega\setminus\{z_1\}$ to the directional derivative of the corresponding continuous harmonic function. Therefore, we have
\begin{eqnarray*}
\lim_{\delta\rightarrow 0}\frac{f(b_x)-f(b_y)}{2\delta^2\cos\theta_w\alpha}=\frac{\partial_1}{\partial\xi_3}\frac{\partial_2}{\partial[i\alpha]} [g_{\Omega}(z_1,z_2)].
\end{eqnarray*}

From (\ref{dre}), we see that $\frac{1}{\delta}\overline{\partial}^{-1}_{\Omega,\delta}(b,v_1)$ is uniformly bounded on a compact subset of $\Omega\setminus \{z_1\}$. By Lemma \ref{dhch}, the directional difference of $\frac{1}{\delta}\overline{\partial}^{-1}(b,v_1)$, with respect to $b$, converges uniformly on a compact subset of $\Omega\setminus\{z_1\}$ to the corresponding directional derivative of the limit continuous harmonic function, then the theorem follows.
\end{proof}

Similarly as in Sect. \ref{nf}, the primal graph $\mathcal{G}_{\Omega,\delta}$ can be extended across $L_{0,\delta}$, by constructing a symmetric graph with respect to the axis $L_{0,\delta}$. We also extend the Green's function $G_{\Omega,\delta}(v_1,v_2)$ on the primal graph across the boundary by letting $G_{\Omega,\delta}(v_1,v_2)=G_{\Omega,\delta}(v_1,v_2^*)$ if $v_2$ and $v_2^*$ are symmetric with respect to $L_{0,\delta}$. 

Let $z_1$ be a point in $\overline{\Omega}$, and $z_2\in L_0$ such that $z_2\neq z_1$. Using the extensions above, we have a discrete harmonic function and $G_{\Omega,\delta}(z_1,\cdot)$, in a neighborhood of $z_2$, and they are uniformly bounded in the neighborhood of $z_2$ with a bound independent of $\delta$. By Lemma \ref{cvhd}, both the harmonic function and its directional derivatives converge uniformly in a neighborhood of $z_2$.

\section{Conformal Invariance}

In this section, we prove an explicit expression for the scaling limit of the discrete holomorphic observable, by utilizing the conformal invariance of the Green's function. 

\begin{lemma}\label{cth}Let \textbf{A}, \textbf{B} be two continuous, (directional) differentiable, complex-valued functions in $\Omega$. Assume for any $z\in\Omega$, $\mathbf{A}(z)$ is perpendicular to $\mathbf{B}(z)$, and there exists two nonparallel directions $\mathbf{\alpha_1}(z)$, $\mathbf{\alpha_2}(z)$, such that
\begin{eqnarray}
\left.\frac{\partial\mathbf{A}}{\partial\mathbf{\alpha_1}}\right|_z&=&\left.\frac{\partial\mathbf{B}}{\partial\mathbf{\beta_1}}\right|_z,\label{aa1bb1}\\
\left.\frac{\partial\mathbf{A}}{\partial\mathbf{\alpha_2}}\right|_z&=&\left.\frac{\partial\mathbf{B}}{\partial\mathbf{\beta_2}}\right|_z,\label{aa2bb2}
\end{eqnarray}
where $\mathbf{\beta_1}=i\mathbf{\alpha_1}$, and $\mathbf{\beta_2}=i\mathbf{\alpha_2}$, then $\mathbf{A+B}$ is (complex) analytic in $\Omega$. Here if $\mathbf{\alpha_1}$ is a unit vector, $\mathbf{A}$ is a complex-valued function, $\frac{\partial \mathbf{A}}{\partial \mathbf{\alpha_1}}$ is defined by
\begin{eqnarray*}
\frac{\partial\mathbf{A}}{\partial\mathbf{\alpha_1}}(z)=\lim_{\epsilon\rightarrow 0}\frac{\mathbf{A}(z+\epsilon\mathbf{\alpha_1})-\mathbf{A}(z)}{\epsilon\mathbf{\alpha_1}}.
\end{eqnarray*}
Morever, for arbitrary directions $\alpha,\beta$, ($\alpha,\beta\in\mathbb{C}$ and $|\alpha|=|\beta|=1$), satisfying $\beta=i\alpha$, we have
\begin{eqnarray*}
\frac{\partial\mathbf{A}}{\partial\alpha}=\frac{\partial\mathbf{B}}{\partial\beta}.
\end{eqnarray*}

\end{lemma}

\begin{proof}Using the chain rule to check the Cauchy-Riemann equation 
\begin{eqnarray*}
&&\Re[\mathbf{A}+\mathbf{B}]_x=\Im[\mathbf{A}+\mathbf{B}]_y\\
&&\Re[\mathbf{A}+\mathbf{B}]_y=-\Im[\mathbf{A}+\mathbf{B}]_x.
\end{eqnarray*}
\end{proof}

\begin{lemma}Define
\begin{eqnarray*}
F_{0,\delta}(w,b)=\left\{\begin{array}{cc}\frac{1}{\delta}\overline{\partial}^{-1}_{\Omega,\delta}(w,b)&\mathrm{if}\ b\in V(\mathcal{G}_{\Omega,\delta})\\0&\mathrm{if}\ b\in V(\mathcal{G}'_{\Omega,\delta})\end{array}\right.\\
F_{1,\delta}(w,b)=\left\{\begin{array}{cc}0&\mathrm{if}\ b\in V(\mathcal{G}_{\Omega,\delta})\\\frac{1}{\delta}\overline{\partial}^{-1}_{\Omega,\delta}(w,b)&\mathrm{if}\ b\in V(\mathcal{G}'_{\Omega,\delta}).\end{array}\right.
\end{eqnarray*}
Then as $\delta\rightarrow 0$, $F_{0,\delta}+F_{1,\delta}$ converges to a meromorphic function $f(z_1,z_2)$ in $z_2$. For each fixed $z_1$, the only pole of $f(z_1,\cdot)$ happens at $z_2=z_1$, with residue $\frac{1}{4\pi}$. Here $z_1$ corresponds to white vertices, and $z_2$ corresponds to black vertices. Moreover, the convergence is uniform on any compact subset of $\Omega\times\Omega\setminus\mathcal{D}$.
\end{lemma}
\begin{proof}Let $\overline{\partial}_0$ be the operator for the whole plane graph $\mathcal{G}_{\delta}^D$, define
\begin{eqnarray*}
J_{0,\delta}(w,b)=\left\{\begin{array}{cc}\frac{1}{\delta}\overline{\partial}^{-1}_{0}(w,b)&\mathrm{if}\ b\in V(\mathcal{G}_{\Omega,\delta})\\0&\mathrm{if}\ b\in V(\mathcal{G}'_{\Omega,\delta})\end{array}\right.\\
J_{1,\delta}(w,b)=\left\{\begin{array}{cc}0&\mathrm{if}\ b\in V(\mathcal{G}_{\Omega,\delta})\\\frac{1}{\delta}\overline{\partial}^{-1}_{0}(w,b)&\mathrm{if}\ b\in V(\mathcal{G}'_{\Omega,\delta}).\end{array}\right.
\end{eqnarray*}
As in the proof of Lemma \ref{db} and Lemma \ref{nb}, (see also \cite{RK02}) explicit computations show that
\begin{eqnarray*}
\lim_{\delta\rightarrow 0}J_{0,\delta}&=&\frac{1}{4\pi\xi_1}\Re\frac{\xi_1}{z_2-z_1}\\
\lim_{\delta\rightarrow 0}J_{1,\delta}&=&\frac{1}{4\pi\xi_3}\Re\frac{\xi_3}{z_2-z_1}
\end{eqnarray*}
Here $\xi_1$ is the unit vector pointing from a white vertex $w$ to an incident vertex in $\mathcal{G}_{\Omega,\delta}$, and $\xi_3$ is the unit vector pointing from a white vertex $w$ to an incident vertex in $\mathcal{G}_{\Omega,\delta}'$, and $\xi_3=i\xi_1$. Hence we have
\begin{eqnarray*}
\lim_{\delta\rightarrow 0}J_{0,\delta}+J_{1,\delta}=\frac{1}{4\pi(z_2-z_1)}
\end{eqnarray*}
Again, according to Lemma \ref{db} and Lemma \ref{nb}, $\lim_{\delta\rightarrow 0}F_{0,\delta}+F_{1,\delta}-J_{0,\delta}-J_{1,\delta}$ is a complex-valued harmonic function in $z_2$ for each fixed $z_1$. Let $b$ be the black vertex in $\mathcal{G}_{\Omega}$ corresponding to $z_2$, we can find two edges incident to $b$ in $\mathcal{G}_{\Omega}$, which are non-parallel and have directions $\alpha_1,\alpha_2$, respectively, such that (\ref{aa1bb1}) and (\ref{aa2bb2}) hold with $\mathrm{A}=\lim_{\delta\rightarrow 0}F_{0,\delta}-J_{0,\delta}$, and $\mathrm{B}=\lim_{\delta\rightarrow 0}F_{1,\delta}-J_{1,\delta}$, therefore $\lim_{\delta\rightarrow 0}F_{0,\delta}+F_{1,\delta}-J_{0,\delta}-J_{1,\delta}$ is analytic in $z_2$, and the lemma is proved.
\end{proof}

Consider $\lim_{\delta\rightarrow 0}(F_{1,\delta}-J_{1,\delta})$. According to Lemma \ref{db},
\begin{eqnarray*}
\lim_{\delta\rightarrow 0}(F_{1,\delta}-J_{1,\delta})=\frac{\partial_1}{\partial\xi_3}\left[g_{\Omega}(z_1,z_2)+\frac{1}{2\pi}\log|z_1-z_2|\right].
\end{eqnarray*}
And according to Lemma \ref{nb}, we have
\begin{eqnarray*}
\frac{\partial_2\lim_{\delta\rightarrow 0}(F_{0,\delta}-J_{0,\delta})}{\partial_2\alpha}=\frac{\partial_1\partial_2}{\partial\xi_3\partial(i\alpha)}\left[g_{\Omega}(z_1,z_2)+\frac{1}{2\pi}\log|z_1-z_2|\right],
\end{eqnarray*}
where $\alpha$ is the unit vector denoting the direction of an edge of $\mathcal{G}_{\Omega,\delta}$ at $z_2$. By Lemma \ref{db}, for each fixed $(z_1,z_2)$, $\lim_{\delta\rightarrow 0}(F_{0,\delta}-J_{0,\delta})$ is perpendicular to $\lim_{\delta\rightarrow 0}(F_{1,\delta}-J_{1,\delta})$, and since each vertex of $\mathcal{G}_{\Omega,\delta}$ is incident to at least two non-parallel edges, there exist two non-parallel directions $\alpha_1$, $\alpha_2$, such that
\begin{eqnarray}
\frac{\partial_2\lim_{\delta\rightarrow 0}(F_{0,\delta}-J_{0,\delta})}{\partial\mathbf{\alpha_1}}&=&\frac{\lim_{\delta\rightarrow 0}(F_{1,\delta}-J_{1,\delta})}{\partial\mathbf{\beta_1}}\label{fjafjb}\\
\frac{\partial_2\lim_{\delta\rightarrow 0}(F_{0,\delta}-J_{0,\delta})}{\partial\mathbf{\alpha_2}}&=&\frac{\lim_{\delta\rightarrow 0}(F_{1,\delta}-J_{1,\delta})}{\partial\mathbf{\beta_2}},
\end{eqnarray}
where $\mathbf{\beta_1}=i\mathbf{\alpha_1}$, and $\mathbf{\beta_2}=i\mathbf{\alpha_2}$. According to Lemma \ref{cth}, for arbitrary directions $\alpha$, $\beta$, satisfying $\beta=i\alpha$, we have
\begin{eqnarray}
\frac{\partial_2\lim_{\delta\rightarrow 0}(F_{0,\delta}-J_{0,\delta})}{\partial\alpha}=\frac{\partial_2\lim_{\delta\rightarrow 0}(F_{1,\delta}-J_{1,\delta})}{\partial\beta}.\label{fafb}
\end{eqnarray}

 Recall that $z_0$ is the limit of the removed vertices $b_{\delta,0}$, as $\delta\rightarrow 0$, and $\lim_{\delta\rightarrow 0}F_{0,\delta}(z_0)=0$. Hence for any smooth path connecting $z_0$ to $z_2$ in $\Omega$, we have
\begin{eqnarray*}
\lim_{\delta\rightarrow 0}(F_{0,\delta}-J_{0,\delta})(z_1,z_2)&=&\int_{z_0}^{z_2}\frac{\partial_2\lim_{\delta\rightarrow 0}(F_{1,\delta}-J_{1,\delta})}{\partial\mathbf{n}}(z_1,\zeta)d\zeta+\lim_{\delta\rightarrow 0}J_{0,\delta}(z_1,z_0)\\
&=&\frac{\partial_1}{\partial\xi_3}\int_{z_0}^{z_2}\frac{\partial_2}{\partial\mathbf{n}}\left[g_{\Omega}(z_1,\zeta)+\frac{1}{2\pi}\log|z_1-\zeta|\right]d\zeta+\frac{1}{4\pi\xi_1}\Re\frac{\xi_1}{z_0-z_1},
\end{eqnarray*}
where $d\zeta$ is the differential along the tangent direction of the path from $z_0$ to $z_2$, and $\frac{\partial_2}{\partial\mathbf{n}}$ is the derivative with respect to the second variable along the normal direction of the path from $z_0$ to $z_2$. In particular we have 
\begin{lemma}
\begin{eqnarray*}
\lim_{\delta\rightarrow 0}F_{0,\delta}(z_1,z_2)=\frac{\partial_1}{\partial\xi_3}\int_{z_0}^{z_2}\frac{\partial_2}{\partial\mathbf{n}}g_{\Omega}(z_1,\zeta)d\zeta,
\end{eqnarray*}
and the integral is along a path from $z_0$ to $z_2$ in a compact subset of $\Omega\setminus\{z_1\}$.
\end{lemma}

Since the limit $F_{0,\delta}$ and $F_{1,\delta}$ also depend on the local direction of the edge corresponding to the white vertex, from now on, we write
\begin{eqnarray*}
f_0^{\Omega}(z_1,z_2,i\xi)&=&\lim_{\delta\rightarrow 0}F_{0,\delta}^{\Omega}(z_1,z_2,\xi)=\frac{\partial_1}{\partial\xi}\int_{z_0}^{z_2}\frac{\partial_2}{\partial\mathbf{n}}g_{\Omega}(z_1,\zeta)d\zeta,\\
f_1^{\Omega}(z_1,z_2,\xi)&=&\lim_{\delta\rightarrow 0}F_{1,\delta}^{\Omega}(z_1,z_2,\xi)=\frac{\partial_1}{\partial\xi}g_{\Omega}(z_1,z_2).
\end{eqnarray*}
where $\xi$ is the direction of the dual edge at the white vertex corresponding to $z_1$, and $i\xi$ is the direction of the primal edge at the white vertex corresponding to $z_1$.

Let $\mathbb{H}=\{z:\Im z>0\}$ be the upper half plane. Let $\phi:\Omega\rightarrow \mathbb{H}$ be the conformal equivalence between $\Omega$ and $\mathbb{H}$, which maps the marked point $z_0$ of $\partial\Omega$ to $\infty$. According to the conformal invariance of the Green's function, we have
\begin{eqnarray*}
g_{\Omega}(z_1,z_2)&=&g_{\mathbb{H}}(\phi(z_1),\phi(z_2))\\
&=&\frac{1}{2\pi}\log\left|\frac{\phi(z_2)-\phi(z_1)}{\phi(z_2)-\overline{\phi(z_1)}}\right|.
\end{eqnarray*}
Define
\begin{eqnarray*}
g_{\Omega}^*(z_1,z_2)=\frac{1}{2\pi}\int_{z_0}^{z_2}\frac{\partial_2}{\partial\mathbf{n}}\log\left|\frac{\phi(\zeta)-\phi(z_1)}{\phi(\zeta)-\overline{\phi(z_1)}}\right|d\zeta,
\end{eqnarray*}
then $g_{\Omega}^*(z_1,\cdot)$ is, locally, $i$ times the harmonic conjugate of $g_{\Omega}(z_1,\cdot)$, with respect to the second variable. Namely,
\begin{eqnarray*}
g^*_{\Omega}(z_1,z_2)=\frac{i}{2\pi}\arg\frac{\phi(z_2)-\phi(z_1)}{\phi(z_2)-\overline{\phi(z_1)}}+C.
\end{eqnarray*}
Note that $g_{\Omega}^{*}(z_1,z_2)$, as a function of $z_2$, is not single-valued. In fact, $g_{\Omega}^*(z_1,\cdot)$ increases by $i$ when winding once counterclockwise around $z_1$. Moreover,
\begin{eqnarray*}
g_{\Omega}(z_1,z_2)+g^*_{\Omega}(z_1,z_2)=\frac{1}{2\pi}\log\left(\frac{\phi(z_2)-\phi(z_1)}{\phi(z_2)-\overline{\phi(z_1)}}\right)+C,
\end{eqnarray*}
which implies, locally, $g_{\Omega}(z_1,z_2)+g_{\Omega}^*(z_1,z_2)$ is (complex) analytic in $z_2$. Similarly we have
\begin{eqnarray*}
g_{\Omega}(z_1,z_2)-g^*_{\Omega}(z_1,z_2)=\frac{1}{2\pi}\log\left(\frac{\overline{\phi(z_2)}-\overline{\phi(z_1)}}{\overline{\phi(z_2)}-\phi(z_1)}\right).
\end{eqnarray*}

Explicit computations show that
\begin{eqnarray*}
f_1^{\Omega}(z_1,z_2,\xi)+f_0^{\Omega}(z_1,z_2,i\xi)&=&\frac{\partial_1}{\partial\xi}(g_{\Omega}(z_1,z_2)+g_{\Omega}^*(z_1,z_2))\\
&=&\frac{1}{2\pi}\left(\frac{\phi'(z_1)}{\phi(z_1)-\phi(z_2)}-\frac{1}{\xi^2}\frac{\overline{\phi'(z_1)}}{\overline{\phi(z_1)}-\phi(z_2)}\right),
\end{eqnarray*}
and
\begin{eqnarray*}
f_1^{\Omega}(z_1,z_2,\xi)-f_0^{\Omega}(z_1,z_2,i\xi)&=&\frac{\partial_1}{\partial\xi}(g_{\Omega}(z_1,z_2)-g_{\Omega}^*(z_1,z_2))\\
&=&\frac{1}{2\pi}\left(\frac{1}{\xi^2}\frac{\overline{\phi'(z_1)}}{\overline{\phi(z_1)}-\overline{\phi(z_2)}}-\frac{\phi'(z_1)}{\phi(z_1)-\overline{\phi(z_2)}}\right).
\end{eqnarray*}

Hence we have
\begin{eqnarray}
f_1^{\Omega}(z_1,z_2,\xi)&=&\frac{1}{4\pi}\left[\left(\frac{\phi'(z_1)}{\phi(z_1)-\phi(z_2)}-\frac{\phi'(z_1)}{\phi(z_1)-\overline{\phi(z_2)}}\right)-\frac{1}{\xi^2}\left(\frac{\overline{\phi'(z_1)}}{\overline{\phi(z_1)}-\phi(z_2)}-\frac{\overline{\phi'(z_1)}}{\overline{\phi(z_1)}-\overline{\phi(z_2)}}\right)\right]\label{lm1}\\
f_0^{\Omega}(z_1,z_2,i\xi)&=&\frac{1}{4\pi}\left[\left(\frac{\phi'(z_1)}{\phi(z_1)-\phi(z_2)}+\frac{\phi'(z_1)}{\phi(z_1)-\overline{\phi(z_2)}}\right)-\frac{1}{\xi^2}\left(\frac{\overline{\phi'(z_1)}}{\overline{\phi(z_1)}-\phi(z_2)}+\frac{\overline{\phi'(z_1)}}{\overline{\phi(z_1)}-\overline{\phi(z_2)}}\right)\right]\label{lm2}
\end{eqnarray}
In particular, in the upper half plane $\mathbb{H}$ we have
\begin{eqnarray*}
f_1^{\mathbb{H}}(z_1,z_2,\xi)&=&\frac{1}{4\pi}\left[\left(\frac{1}{z_1-z_2}-\frac{1}{z_1-\overline{z_2}}\right)-\frac{1}{\xi^2}\left(\frac{1}{\overline{z_1}-z_2}-\frac{1}{\overline{z_1}-\overline{z_2}}\right)\right]\\
f_0^{\mathbb{H}}(z_1,z_2,i\xi)&=&\frac{1}{4\pi}\left[\left(\frac{1}{z_1-z_2}+\frac{1}{z_1-\overline{z_2}}\right)-\frac{1}{\xi^2}\left(\frac{1}{\overline{z_1}-z_2}+\frac{1}{\overline{z_1}-\overline{z_2}}\right)\right]
\end{eqnarray*}

\section{Height function and Gaussian free field}

In this section we study the height functions for perfect matchings on isoradial graphs $\{\mathcal{G}_{\Omega,\delta}^D\}_{\delta}$, and prove that the distribution of height functions converges to GFF in the scaling limit.

Recall that the graph $\mathcal{G}_{\Omega,\delta}^D$, constructed by the superposition of an isoradial graph $\mathcal{G}_{\Omega,\delta}$ and its interior dual graph $\mathcal{G}'_{\Omega,\delta}$, and then the removal of a vertex of $\mathcal{G}_{\Omega,\delta}$ on the boundary, is a bipartite, isoradial graph which admits a perfect matching, as discussed in Sect.~\ref{da}. 

The probability measure for perfect matchings on the graph $\mathcal{G}_{\Omega,\delta}^D$ is defined to be proportional to the product of critical edge weights, see (\ref{pm}). The height function is defined as in (\ref{dh1})-(\ref{hetd}).

Assume $g_{1,\delta},...,g_{k,\delta}$ are $k$ distinct faces of $\mathcal{G}_{\Omega,\delta}^D$. Recall that $\mathcal{G}_{\Omega,\delta}^D$ is obtained by superimposing $\mathcal{G}_{\Omega,\delta}$ and its interior dual graph $\mathcal{G}'_{\Omega,\delta}$, then removing a vertex of $\mathcal{G}_{\Omega,\delta}$ on the boundary. Also recall that $\mathcal{G}_{\Omega,\delta}$ is the interior dual graph of $\mathcal{G}''_{\Omega,\delta}$, and the boundary of $\mathcal{G}''_{\Omega,\delta}$ has a straight portion $L_{0,\delta}$. Let $f_{1,\delta},...,f_{k,\delta}$ be $k$ distinct faces outside $\mathcal{G}_{\Omega,\delta}^D$, but incident to the boundary of $\mathcal{G}_{\Omega,\delta}^D$ and $L_{0,\delta}$. Assume
\begin{eqnarray*}
\lim_{\delta\rightarrow 0}g_{j,\delta}&=&z_j\\
\lim_{\delta\rightarrow 0}f_{j,\delta}&=&w_j\qquad \forall 1\leq j\leq k
\end{eqnarray*}
where $\{z_j\}_{j=1}^{k}$ are distinct interior points of $\Omega$, and $\{w_j\}_{j=1}^{k}$ are distinct points on the flat boundary $L_0$ of $\partial \Omega$.

Let $h_{j,\delta}=h(g_{j,\delta})-h(f_{j,\delta})$ be the height difference between the faces $g_{j,\delta}$ and $f_{j,\delta}$. Let $\overline{h}_{j,\delta}$ be the mean value of $h_{j,\delta}$, i.e., $\overline{h}_{j,\delta}=\mathbb{E}h_{j,\delta}$.

If we fix the height $h(f_{1,\delta})=0$, it is not hard to see that for any given dimer configuration on $\mathcal{G}_{\Omega,\delta}^D$, $h(f_{j,\delta})$ is deterministic and identically 0, according to the definition of height function in Sect.~\ref{dfht}. Therefore,
\begin{eqnarray*}
h_{j,\delta}=h(g_{j,\delta}).
\end{eqnarray*}

 For each $\delta$ sufficiently small,  let $\gamma_{j}^{\delta}, (1\leq j\leq k)$ be pairwise disjoint dual paths of $\mathcal{G}_{\Omega,\delta}^D$, starting at $f_{j,\delta}$, ending on the face $g_{j,\delta}$ and consisting of dual edges of $\mathcal{G}_{\Omega,\delta}^D$.

Moreover, according to our specific definition of height function in Sect.~\ref{dfht},
\begin{eqnarray*}
\overline{h}_{j,\delta}=0\qquad\forall 1\leq j\leq k.
\end{eqnarray*}
For notational simplicity, we use $\gamma_j(1\leq j\leq k)$ instead of $\gamma_j^{\delta}$. For each fixed $j (1\leq j\leq k)$,  we classify all the edges of $\mathcal{G}_{\Omega,\delta}^D$  crossed by $\gamma_j$ into 4 different types as follows
\begin{enumerate}
\item the edge has the white vertex on the left of $\gamma_j$, and the black vertex is a vertex of $\mathcal{G}_{\Omega,\delta}$, denote the set of all such edges by $U_{j,1}$.
\item the edge has the white vertex on the left of $\gamma_j$, and the black vertex is a vertex of $\mathcal{G}_{\Omega,\delta}'$, denote the set of all such edges by $U_{j,2}$.
\item the edge has the black vertex on the left of $\gamma_j$, and the black vertex is a vertex of $\mathcal{G}_{\Omega,\delta}$, denote the set of all such edges  by $U_{j,3}$.
\item the edge has the black vertex on the left of $\gamma_j$, and the black vertex is a vertex of $\mathcal{G}_{\Omega,\delta}'$, denote the set of all such edges by $U_{j,4}$.
\end{enumerate}
Then
\begin{eqnarray}
&&\mathbb{E}h_{1,\delta}\cdot...\cdot h_{k,\delta}\label{ht}\label{mh}\\
&=&\sum_{j_1,...,j_k\in\{1,2,3,4\}}(-1)^{\xi(j_1)+...+\xi(j_k)}\sum_{t_1,...,t_k,e_{s,j_s,t_s}\in U_{s,j_s}}\mathbb{E}(\mathbb{I}(e_{1,j_1,t_1})-\mathbb{P}(e_{1,j_1,t_1}))\cdot...\cdot(\mathbb{I}(e_{k,j_k,t_k})-\mathbb{P}(e_{k,j_k,t_k})),\label{htt}
\end{eqnarray}
where $\xi$ is a function defined on $\{1,2,3,4\}$, such that $\xi(1)=\xi(2)=0$, and $\xi(3)=\xi(4)=1$. For each $s$, $1\leq s\leq k$, $t_s$ is an integer indexing the $t_s$-th edge crossed by $\gamma_s$ in the class $U_{s,j_s}$, counted starting from the face $f_{j,\delta}$.  The $t_s$-th edge in $U_{s,j_s}$ is denoted by $e_{s,j_s,t_s}$. The expression of $\mathbb{E}h_{1,\delta}\cdot...\cdot h_{k,\delta}$ as in (\ref{htt}) follows from applying the formula (\ref{hetd}) to compute height functions and expanding the product out.

\begin{lemma}Let $e_i=(w_i,b_i)$, for $i=1,...,n$, be a set of $n$ distinct edges in $\mathcal{G}_{\Omega,\delta}^D$, then 
\begin{eqnarray}
&&\mathbb{E}(\mathbb{I}(e_1)-\mathbb{P}(e_1))\cdot...\cdot(\mathbb{I}(e_n)-\mathbb{P}(e_n))\label{jd}\\
&=&\prod_{i=1}^{n}\overline{\partial}_{\Omega,\delta}(w_i,b_i)\det\left(\begin{array}{cccc}0&\overline{\partial}_{\Omega,\delta}^{-1}(w_1,b_2)&...&\overline{\partial}_{\Omega,\delta}^{-1}(w_1,b_n)\\\overline{\partial}_{\Omega,\delta}^{-1}(w_2,b_1)&0&...&\overline{\partial}_{\Omega,\delta}^{-1}(w_2,b_n)\\ \overline{\partial}_{\Omega,\delta}^{-1}(w_n,b_1)&\overline{\partial}_{\Omega,\delta}^{-1}(w_n,b_2)&...&0 \end{array}\right)\label{ajp}
\end{eqnarray}
Here $\overline{\partial}$ is a special (complex) weighted adjacent matrix with respect to the isoradial graph $\mathcal{G}_{\Omega,\delta}^D$, as discussed in Sect. \ref{cvdo}, and $\overline{\partial}^{-1}$ is the inverse matrix of $\overline{\partial}$.
\end{lemma}
\begin{proof}Apply the same technique as in Lemma 21 of \cite{RK00} to the setting of isoradial graphs.
\end{proof}

Recall that $f_0(w,b,i\xi)=\lim_{\delta\rightarrow 0}\frac{1}{\delta}\overline{\partial}_{\Omega,\delta}^{-1}(w,b)$, if $b$ is a vertex of $\mathcal{G}_{\Omega,\delta}$, and $\xi$ is the direction of the dual edge of $\mathcal{G}'_{\Omega,\delta}$ at $w$. Also recall that $f_1(w,b,\xi)=\lim_{\delta\rightarrow 0}\frac{1}{\delta}\overline{\partial}_{\Omega,\delta}^{-1}(w,b)$, if $b$ is a vertex of $\mathcal{G}'_{\Omega,\delta}$, and $\xi$ is the direction  of the dual edge of $\mathcal{G}'_{\Omega,\delta}$ at $w$. Note that the convergence is uniform on any compact subset of $\Omega\times\Omega\setminus\mathcal{D}$.

Plugging (\ref{ajp}) to (\ref{htt}), and expanding the determinant out, we obtain an expression of $\mathbb{E}h_{1,\delta}\cdot...\cdot h_{k,\delta}$ in terms of entries of $\overline{\partial}$ and $\overline{\partial}^{-1}$.  A typical term in this expression is
\begin{eqnarray}
(-1)^{\xi(j_1)+...+\xi(j_k)}\prod_{i=1}^{k}\overline{\partial}_{\Omega,\delta}(w_i,b_i)\mathrm{sgn}(\sigma)\overline{\partial}_{\Omega,\delta}^{-1}(w_1,b_{\sigma(1)})\overline{\partial}_{\Omega,\delta}^{-1}(w_2,b_{\sigma(2)})\cdot...\cdot\overline{\partial}_{\Omega,\delta}^{-1}(w_{k},b_{\sigma(k)}).\label{to}
\end{eqnarray}
Here $\sigma$ is a permutation of $k$ elements with no fixed point. Let us first assume that $\sigma$ is a $k$-cycle, reorder the indices so that (\ref{to}) becomes
\begin{eqnarray}
&&(-1)^{\xi(j_1)+...+\xi(j_k)}\prod_{i=1}^{k}\overline{\partial}_{\Omega,\delta}(w_i,b_i)\mathrm{sgn}(\sigma)\overline{\partial}_{\Omega,\delta}^{-1}(w_1,b_2)\overline{\partial}_{\Omega,\delta}^{-1}(w_2,b_3)\cdot...\cdot\overline{\partial}_{\Omega,\delta}^{-1}(w_{k},b_1)\label{otm}\\
&=&(-1)^{\xi(j_1)+...+\xi(j_k)}\mathrm{sgn}(\sigma)\delta^{k}\prod_{i=1}^{k}\overline{\partial}_{\Omega,\delta}(w_i,b_i)f_{\eta(j_2)}(w_1,b_2,\xi_1^{\eta(j_2)})\cdot f_{\eta(j_3)}(w_2,b_3,\xi_2^{\eta(j_3)})\cdot\notag\\
&&...\cdot f_{\eta(j_k)}(w_{k-1},b_k,\xi_{k-1}^{\eta(j_k)})\cdot f_{\eta(j_1)}(w_k,b_1,\xi_k^{\eta(j_1)})+o(1),\label{lmo}
\end{eqnarray}
where $\eta$ is a function on $\{1,2,3,4\}$, such that $\eta(1)=\eta(3)=0$, and $\eta(2)=\eta(4)=1$. $\xi_i^{0}$ is the direction of the primal edge at $w_i$, while $\xi_i^{1}$ is the direction of the dual edge at $w_i$. Recall that $j_i$ ($1\leq i\leq k$) is an integer indicating the type of the edge $w_ib_i$, as discussed before. In particular, $j_i=1,3$, means that $b_i$ is a vertex of $\mathcal{G}_{\Omega,\delta}$, and $j_i=2,4$, means that $b_i$ is a vertex of $\mathcal{G}'_{\Omega,\delta}$.

Lemma \ref{dz} below is contained in \cite{BdT07}, we present the proof here so that the paper is self-contained.

\begin{lemma}\label{dz}For $1\leq i\leq k$, when traveling along $\gamma_i$ from $f_{i,\delta}$ to $g_{i,\delta}$, let $\Delta z_i$ be the increment along the dual edge of $w_ib_i$, where the dual edge is in $\gamma_i$, and $w_ib_i$ is the unique primal edge crossing the dual edge. We have
\begin{eqnarray*}
(-1)^{\xi(j_1)+...+\xi(j_k)}\prod_{i=1}^{k}\overline{\partial}_{\Omega,\delta}(w_i,b_i)\delta^k=(-i)^{k}\Delta z_1\cdot...\cdot \Delta z_k
\end{eqnarray*}
\end{lemma}
\begin{proof}  According to the local geometry of the isoradial graph, we  know that $\Delta z_i$ has length equal to $2\delta \sin \theta_i$, and direction perpendicular to $w_ib_i$, here $\theta_i$ is the rhombus half-angle corresponding to $w_ib_i$

By definition of the $\overline{\partial}$ operator, for $1\leq i\leq k$, $\overline{\partial}_{\Omega,\delta}(w_i,b_i)$ is the complex number of length $2\sin\theta_i$,  with direction pointing from $w_i$ to $b_i$. Also recall that $w_ib_i$ is an edge of $\mathcal{G}_{\Omega,\delta}^D$ intersecting $\gamma_i$. If $w_i$ is on the left of the path $\gamma_i$ when travelling along $\gamma_i$ from $f_{i,\delta}$ to $g_{i,\delta}$, then
\begin{eqnarray*}
i\delta\overline{\partial}_{\Omega,\delta}(w_i,b_i)=\Delta z_i,
\end{eqnarray*}
however, if $w_i$ is on the right of the path $\gamma_i$ oriented from $f_{i,\delta}$ to $g_{i,\delta}$, then
\begin{eqnarray*}
(-i)\delta\overline{\partial}_{\Omega,\delta}(w_i,b_i)=\Delta z_i,
\end{eqnarray*}
hence the lemma is proved.
\end{proof}

Applying Lemma \ref{dz} to (\ref{lmo}), we obtain another expression of (\ref{otm}), namely
\begin{eqnarray}
(\ref{otm})=\mathrm{sgn}(\sigma)\prod_{i=1}^{k}\left[\Delta z_i f_{\eta(j_{i+1})}\left(w_i,b_{i+1},\xi_i^{\eta(j_{i+1})}\right)\right]\label{aep}
\end{eqnarray}

\begin{lemma}\label{lma}For $1\leq i\leq k$
\begin{eqnarray}
&&\Delta z_i f_{\eta(j_{i+1})}\left(w_i,b_{i+1},\xi_i^{\eta(j_{i+1})}\right)\label{ss}\label{ll1}\\
&=&\frac{1}{4\pi}\left[\left(\frac{\phi'(w_i)}{\phi(w_i)-\phi(b_{i+1})}+(-1)^{\eta(j_{i+1})}\frac{\phi'(w_i)}{\phi(w_i)-\overline{\phi(b_{i+1})}}\right)\Delta z_i\right.\\
&&\left.+(-1)^{\eta(j_i)+1}\overline{\Delta z_i}\left(\frac{\overline{\phi'(w_i)}}{\overline{\phi(w_i)}-\phi(b_{i+1})}+(-1)^{\eta(j_{i+1})}\frac{\overline{\phi'(w_i)}}{\overline{\phi(w_i)}-\overline{\phi(b_{i+1})}}\right)\right]\label{ll3}
\end{eqnarray}
\end{lemma}

\begin{proof}We plug in formulas (\ref{lm1}), (\ref{lm2}) to (\ref{ss}), and note that $\eta_{j_{i+1}}=0$ (resp. $\eta_{j_{i+1}}=1$), if $b_{i+1}$ is a vertex of $\mathcal{G}_{\Omega,\delta}$  (resp. $\mathcal{G}'_{\Omega,\delta}$). 

Moreover, recall that $w_ib_i$ is an edge of $\mathcal{G}_{\Omega,\delta}^D$ crossed by $\gamma_i$. An edge $\Delta z_i$ along $\gamma_i$ is always perpendicular to the crossing edge $w_ib_i$ in $\mathcal{G}_{\Omega,\delta}^D$.  If $\eta(j_i)=0$, $b_i$ is a vertex of $\mathcal{G}_{\Omega,\delta}$, $w_ib_i$ is parallel to $i\xi$, and therefore $\Delta z_i$ is parallels to $\xi$, hence we have
\begin{eqnarray*}
\frac{\Delta z_i}{\xi^2}=\overline{\Delta z_i},\qquad \mathrm{if}\ \eta(j_i)=0.
\end{eqnarray*}

Similarly, if $\eta(j_i)=1$, $b_i$ is a vertex of $\mathcal{G}'_{\Omega,\delta}$, $w_ib_i$ is parallel to $\xi$, and therefore $\Delta z_i$ is parallels to $i\xi$, hence we have
\begin{eqnarray*}
\frac{\Delta z_i}{(i\xi)^2}=\overline{\Delta z_i},\qquad \mathrm{if}\ \eta(j_i)=0.
\end{eqnarray*}
Then the lemma follows.

\end{proof}

We plug in (\ref{ll1})-(\ref{ll3}) to (\ref{aep}), to obtain an expression of (\ref{otm}), then plug in the expression of (\ref{otm}) to (\ref{ajp}) to obtain an expression of (\ref{jd}); then plug in the expression of (\ref{jd}) to (\ref{htt}) to compute the height moment (\ref{mh}).

For $1\leq s\leq k$, each factor $f_{\eta(j_{s+1})}(w_{s},b_{s+1},\xi_s^{\eta(j_{s+1})})$ in (\ref{lmo}) is the sum of four terms. We expand them out and express (\ref{lmo}) as the sum of $4^k$ terms.  

 Let $S={1,2,...,k}$. Each one of the $4^k$ terms in the expansion of (\ref{lmo}) gives a division of $S$ into 4 disjoint subsets $S_1,S_2,S_3,$ and $S_4$, such that 
\begin{enumerate} 
\item $S_1\cup S_2\cup S_3\cup S_4=S$, 
\item $S_i\cap S_j=\emptyset$, if $i,j\in\{1,2,3,4\}$ and $i\neq j$.
 \end{enumerate}
 
 Namely, if  $\phi(w_{i,t_i})$ and $\phi(b_{i,t_i})$ appear in the term, then $i\in S_1$; if $\phi(w_{i,t_i})$ and $\overline{\phi(b_{i,t_i})}$ appear in the term, then $i\in S_2$; if $\overline{\phi(w_{i,t_i})}$ and $\phi(b_{i,t_i})$ appear in the term, then $i\in S_3$; if $\overline{\phi(w_{i,t_i})}$ and $\overline{\phi(b_{i,t_i})}$ appear in the term, then $i\in S_4$. 

We claim that if in the division of $S$ determined by a term in the expansion of (\ref{lmo}), $S_2\cup S_3\neq \emptyset$, then the term, summing over $t_1,...,t_k$ will go to $0$ as $\delta\rightarrow 0$. In fact, applying Lemmas \ref{dz} and \ref{lma} , a term in the expansion of (\ref{lmo}) can be written as
\begin{eqnarray*}
&&(-i)^{k}\mathrm{sgn}(\sigma)u_{\sigma,S_1,S_2,S_3,S_4}(z_{1,t_1},...,z_{k,t_k},\overline{z}_{1,t_1},...,\overline{z}_{k,t_k})\prod_{i\in S_1}\left[\Delta  z_{i,t_i}\right]\prod_{i\in S_4}\overline{\left[-\Delta z_{i,t_i}\right]}\\
&&\prod_{i\in S_2}\left[(-1)^{\eta(j_i)}\Delta z_{i,t_i} \right]\prod_{i\in S_3}\left[(-1)^{\eta(j_i)+1}\overline{\Delta z_{i,t_i}}\right],
\end{eqnarray*}
where $z_{i,t_i}=\frac{w_{i,t_i}+b_{i,t_i}}{2}$, and $u_{\sigma,S_1,S_2,S_3,S_4}$ is a smooth function when $z_1,...,z_k$ are distinct. Note also that the function $u_{\sigma,S_1,S_2,S_3,S_4}$ depends on the division $S_1,S_2,S_3,S_4$, i.e.  different monomials in the expansion of (\ref{lmo}) give different function. It also depends on the specific permutation in the expansion of the determinant in (\ref{ajp}).

Here $\Delta z_{i,t_i}$ is the $t_i$-th oriented dual edge connecting adjacent faces along the path $\gamma_i$, where edges along $\gamma_i$ are oriented in such a way that $\gamma_i$ is oriented from $f_{i,\delta}$ to $g_{i,\delta}$.

\begin{lemma}
\begin{eqnarray}
\sum_{t_i}(-1)^{\eta(j_i)}\Delta z_{i,t_i}=O(\delta)\label{vnh}
\end{eqnarray}
\end{lemma}
\begin{proof}Let $F_0(:=f_{i,\delta})$, $F_1$,..., $F_{T_i}(:=g_{i,\delta})$ be all the faces along the dual path $\gamma_i$, then
\begin{eqnarray*}
\Delta z_{i,t_i}=F_{t_i}-F_{t_i-1}, \qquad\qquad\mathrm{for}\ 1\leq t_i\leq T_i,
\end{eqnarray*}
where $F_{t_i}-F_{t_i-1}$ is a dual edge, and it is crossed by an unique primal edge $w_{i,t_i}b_{i,t_i}$. On the left hand side of (\ref{vnh}), the sign of $\Delta z_{i,t_i}$ depends on the type of $b_{i,t_i}$; namely, if $b_{i,t_i}$ is a vertex of $\mathcal{G}_{\Omega,\delta}$ (resp. $\mathcal{G}'_{\Omega,\delta}$), then the sign of $\Delta z_{i,t_i}$ in the sum (\ref{vnh}) is positive (resp. negative).

Note that each face in $\mathcal{G}_{\Omega,\delta}^D$ has 4 vertices, exactly two vertices are black, and 2 vertices are white. Moreover, in each pair of black vertices of the same face in $\mathcal{G}_{\Omega,\delta}^D$, exactly one black vertex is a vertex of $\mathcal{G}_{\Omega,\delta}$, and the other is a vertex of $\mathcal{G}'_{\Omega,\delta}$.

Moving along dual edges of $\gamma_i$, at the beginning we meet with dual edges whose crossing primal edges share a black vertex $b_1$ until the $s_1$-th dual edge, next we meet with the $(s_1+1)$-th dual edge whose crossing primal edge has a different black vertex $b_2$. In fact, $b_2$ and $b_1$ must share a face, and exactly one of them is a vertex of $\mathcal{G}_{\Omega,\delta}$, and the other is a vertex of $\mathcal{G}'_{\Omega,\delta}$.

Then we keep moving along $\gamma_i$, assume all the dual edges from the $(s_1+1)$-th dual edge to the $s_2$th dual edge, have crossing primal edges sharing a black vertex $b_2$, but the $(s_2+1)$-th dual edge has a crossing primal edge with a different black vertex $b_3$. Evidently $b_2$ and $b_3$ share a face in $\mathcal{G}_{\Omega,\delta}^D$, and they are black vertices of different types. 

We continue this process and obtain
\begin{eqnarray*}
\sum_{t_i}(-1)^{\eta(j_i)}\Delta z_{i,t_i}=\sum_{p=1}^{q}(-1)^{p+1}[F_{s_p}-F_{s_{p-1}}]
\end{eqnarray*}
where for $p=1,...,q-1$, $F_{s_p}$'s are all the faces separating two dual edges along $\gamma_i$ whose crossing primal edges have different black vertices. 

For $p=1,...,q$, let $b_p$ be the common black vertex shared by crossing primal edges of all the dual edges along $\gamma_i$ between $F_{s_{p-1}}$ and $F_{s_p}$, we have
\begin{eqnarray*}
\sum_{t_i}(-1)^{\eta(j_i)}\Delta z_{i,t_i}=\sum_{p=1}^{q}(-1)^{p+1}[(F_{s_p}-b_p)+(b_p-F_{s_{p-1}})].
\end{eqnarray*}
Moreover, the geometry of the isoradial graph gives us
\begin{eqnarray*}
2F_{s_{p}}=b_p+b_{p+1},\qquad\qquad\mathrm{for}\ 1\leq p\leq q-1.
\end{eqnarray*}
see Figure \ref{fc}.

\begin{figure}[htb]
\begin{center}
\includegraphics{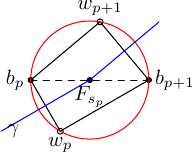}
\caption{face of isoradial graph and dual path}\label{fc}
\end{center}
\end{figure}

As a result,
\begin{eqnarray*}
\sum_{t_i}(-1)^{\eta(j_i)}\Delta z_{i,t_i}=(-1)^{q+1}(F_{s_q}-b_q)-(F_0-b_1)=O(\delta).
\end{eqnarray*}
\end{proof}

Therefore if a term in the expansion of (\ref{lmo}) has a division of $S$ in which $S_2\cup S_3\neq \emptyset$, then the sum over $t_1,...,t_k$ of the term will go to zero as $\delta\rightarrow 0$, since it is the sum of a smooth function multiplying an oscillating term. Taking into account the boundary behavior of $\overline{\partial}^{-1}_{\Omega,\delta}$ near $L_0$, as discussed in  Sect.~\ref{nf}, we have
\begin{eqnarray}\label{mf}
\lim_{\delta\rightarrow 0}\mathbb{E}h_{1,\delta}\cdot...\cdot h_{k,\delta}=\sum_{\epsilon_1,...,\epsilon_k\in\{-1,1\}}\left(\frac{-i}{4\pi}\right)^{k}\prod_{i=1}^{k}\epsilon_{i}\int_{\phi(\gamma_1),...,\phi(\gamma_k)}\det U([\phi(z_1)]^{(\epsilon_1)},...,[\phi(z_k)]^{(\epsilon_k)}) dz_1^{(\epsilon_1)}\cdot...\cdot d z_k^{(\epsilon_k)}
\end{eqnarray}
where $U(x_1,...,x_k)$ is a $k\times k$ matrix defined by $U_{i,i}=0$, and $U_{i,j}=\frac{1}{x_i-x_j}$; if $w\in \mathbb{C}$ is a complex number, we use the notation $w^{(1)}:=w$, and $w^{(-1)}:=\overline{w}$.

Lemma \ref{ctl} below is a classical formula to compute the Cauchy determinant.

\begin{lemma}(\cite{Kr})\label{ctl}Let $M=(m_{ij})$ be the $k\times k$ matrix defined by $m_{ii}=0$, and $m_{ij}=\frac{1}{x_i-x_j}$. When $k$ is odd, $\det M=0$, and when $k$ is even, 
\begin{eqnarray*}
\det M=\sum\frac{1}{(x_{\sigma(1)}-x_{\sigma(2)})^2(x_{\sigma(3)}-x_{\sigma(4)})^2\cdots(x_{\sigma(k-1)}-x_{\sigma(k)})^2},
\end{eqnarray*}
where the sum is over all $(k-1)!!$ possible pairings $\{\{\sigma(1),\sigma(2)\},...,\{\sigma(k-1),\sigma(k)\}\}$ of $\{1,...,k\}$.
\end{lemma} 

Now we apply the moment formula (\ref{mf}) and compute the second moment explicitly. We have
\begin{eqnarray*}
&&\lim_{\delta\rightarrow 0}\mathbb{E}h_{1,\delta}(z_1)h_{2,\delta}(z_2)\\
&=&-\frac{1}{16\pi^2}\left[\int_{\phi(\gamma_1),\phi(\gamma_2)}\frac{1}{(\zeta_1-\zeta_2)^2}d\zeta_1d\zeta_2-\int_{\phi(\gamma_1),\phi(\gamma_2)}\frac{1}{(\zeta_1-\overline{\zeta}_2)^2}d\zeta_1d\overline{\zeta}_2\right.\\
&&\left.-\int_{\phi(\gamma_1),\phi(\gamma_2)}\frac{1}{(\overline{\zeta}_1-\zeta_2)^2}d\overline{\zeta}_1d\zeta_2+\int_{\phi(\gamma_1),\phi(\gamma_2)}\frac{1}{(\overline{\zeta}_1-\overline{\zeta}_2)^2}d\overline{\zeta}_1d\overline{\zeta}_2\right]\\
&=&-\frac{1}{8\pi^2}\log\left|\frac{\phi(z_2)-\phi(z_1)}{\phi(z_2)-\overline{\phi}(z_1)}\right|=-\frac{1}{4\pi}g_{\mathbb{H}}(\phi(z_1),\phi(z_2)),
\end{eqnarray*}
where $g_{\mathbb{H}}$ is the Green's function for the upper half plane $\mathbb{H}$. 

Assume $z_1,...,z_k$ are distinct points in $\Omega$. Applying Lemma \ref{ctl}, we have, when $k$ is odd,
\begin{eqnarray*}
\lim_{\delta\rightarrow 0}\mathbb{E}[h_{1,\delta}(z_1)\cdot...\cdot h_{k,\delta}(z_k)]=0.
\end{eqnarray*}
When $k$ is even,
\begin{eqnarray*}
\lim_{\delta\rightarrow 0}\mathbb{E}[h_{1,\delta}(z_1)\cdot...\cdot h_{k,\delta}(z_k)]=\sum_{pairings\  \sigma}\left(-\frac{1}{4\pi}\right)^{\frac{k}{2}}\prod_{j=1}^{\frac{k}{2}}g_{\mathbb{H}}(\phi(z_{\sigma(2j-1)}),\phi(z_{\sigma(2j)})).
\end{eqnarray*}

The next lemma gives a control on the expectation of product of height functions at finitely many interior  points of the finite domain, of which  two or more points are identical. For a similar bound for height moments of dimer configurations on a whole-plane isoradial graph, see Lemma 20 of \cite{BdT07}.

\begin{lemma}\label{sg}
If two or more $z_i$'s are equal, we have
\begin{eqnarray*}
\mathbb{E}[h_{1,\delta}(z_1)\cdot...\cdot h_{k,\delta}(z_k)]=O([\log\delta]^{\ell})
\end{eqnarray*}
where $\ell$ is the number of coincides, i.e. $k-\ell$ is the number of distinct $z_i$'s.
\end{lemma}

\begin{proof}If two or more $z_i$'s are equal, we choose the paths $\gamma_{i,\delta}$'s, consisting of dual edges of $\mathcal{G}_{\Omega,\delta}^D$, so that the $\gamma_{i,\delta}'s$ for different $i$'s are close to each other only at coinciding $z_i$'s, but far away from each other elsewhere. Previous computations show that
\begin{eqnarray}
&&\mathbb{E}h_{1,\delta}\cdot...\cdot h_{k,\delta}\\
&=&\sum_{j_1,...,j_k\in\{1,2,3,4\}}\sum_{\sigma\in \hat{S}_k}\sum_{t_1,...,t_k}\left(\frac{-i}{4\pi}\right)^k\mathrm{sgn}(\sigma)\\
&&\prod_{i=1}^{k}\left[\mathrm{Error\ terms}+\left(\frac{\phi'(w_{i,t_i,j_i})}{\phi(w_{i,t_i,j_i})-\phi(b_{\sigma(i),t_{\sigma(i)},j_{\sigma(i)}})}+(-1)^{\eta(j_{\sigma(i)})}\frac{\phi'(w_{i,t_i,j_i})}{\phi(w_{i,t_i,j_i})-\overline{\phi(b_{\sigma(i),t_{\sigma(i)},j_{\sigma(i)}})}}\right)\Delta z_{i,t_i,j_i}\right.\\
&&+\left.(-1)^{\eta(j_{\sigma(i)})+1}\overline{\Delta z_{i,t_i,j_i}}\left(\frac{\overline{\phi'(w_i,t_i,j_i)}}{\overline{\phi_{w_i,t_i,j_i}}-\phi(b_{\sigma(i),t_{\sigma(i)},j_{\sigma(i)}})}+(-1)^{\eta(j_{\sigma(i)})}\frac{\overline{\phi'(w_{i,t_i,j_i})}}{\overline{\phi(w_{i,t_i,j_i})}-\overline{\phi(b_{\sigma(i),t_{\sigma(i)},j_{\sigma(i)}})}}\right).
\right]\label{esm}
\end{eqnarray}
Here $\hat{S}_k$ is the set of all permutations with $k$ elements with no fixed point.

We need to estimate the error terms in (\ref{esm}). When $b_{\sigma(i),t_{\sigma(i)},j_{\sigma(i)}}$ is a vertex of $\mathcal{G}'_{\Omega,\delta}$, from the proof of Lemma \ref{db}, we know that
\begin{eqnarray*}
|\mathrm{Error\ terms}|\leq O\left(\frac{\delta^2}{|w_{i,t_i,j_i}-b_{\sigma(i),t_{\sigma(i)},j_{\sigma(i)}}|}\right)+o(\delta)
\end{eqnarray*}
When $b_{\sigma(i),t_{\sigma(i)},j_{\sigma(i)}}$ is a vertex of $\mathcal{G}_{\Omega,\delta}$, from the proof of Lemma \ref{nb}, we know that
\begin{eqnarray*}
|\mathrm{Error\ terms}|\leq O\left(\frac{\delta}{|w_{i,t_i,j_i}-b_{\sigma(i),t_{\sigma(i)},j_{\sigma(i)}}|}\right)+o(\delta)
\end{eqnarray*}

 If $\ell$ is the number of coincides of $z_i$'s, we have
 \begin{eqnarray*}
&&\left| \mathbb{E}h_{1,\delta}\cdot...\cdot h_{k,\delta}-\sum_{j_1,...,j_k\in\{1,2,3,4\}}
\sum_{\sigma\in \hat{S}_k}\sum_{t_1,...,t_k}\left(\frac{-i}{4\pi}\right)^k\mathrm{sgn}(\sigma)\right.\\
&&\prod_{i=1}^{k}\left[\left(\frac{\phi'(w_{i,t_i,j_i})}{\phi(w_{i,t_i,j_i})-\phi(b_{\sigma(i),t_{\sigma(i)},j_{\sigma(i)}})}+(-1)^{\eta(j_{\sigma(i)})}\frac{\phi'(w_{i,t_i,j_i})}{\phi(w_{i,t_i,j_i})-\overline{\phi(b_{\sigma(i),t_{\sigma(i)},j_{\sigma(i)}})}}\right)\Delta z_{i,t_i,j_i}\right.\\
&&+\left.\left.(-1)^{\eta(j_{\sigma(i)})+1}\overline{\Delta z_{i,t_i,j_i}}\left(\frac{\overline{\phi'(w_i,t_i,j_i)}}{\overline{\phi_{w_i,t_i,j_i}}-\phi(b_{\sigma(i),t_{\sigma(i)},j_{\sigma(i)}})}+(-1)^{\eta(j_{\sigma(i)})}\frac{\overline{\phi'(w_{i,t_i,j_i})}}{\overline{\phi(w_{i,t_i,j_i})}-\overline{\phi(b_{\sigma(i),t_{\sigma(i)},j_{\sigma(i)}})}}\right)
\right]\right|\\
&&\leq O[(\log\delta)^{\ell}],
 \end{eqnarray*}
 then the lemma follows.
\end{proof}

We have the following theorem:
 
 \begin{theorem}The scaling limit of distribution of (unnormalized) height is conformally invariant and universal, and is the same as $\frac{1}{2\sqrt{\pi}}$ times a Gaussian free field.
 \end{theorem}
 \begin{proof}It suffices to prove that for any $\psi\in C_0^{\infty}(\Omega)$,
 \begin{eqnarray}
 \sum_{f\in\mathcal{G}_{\Omega,\delta}^D}\psi(f)h_{\delta}(f)A(f)\rightarrow \int_{\Omega}\mathcal{F}\psi dxdy,\label{wkc}
 \end{eqnarray}
 where the convergence is in distribution, $f$ is a face in $\mathcal{G}_{\Omega,\delta}^D$, and $\mathcal{F}$ is the GFF for the domain $\Omega$, $A(f)$ is the area of the face $f$. We have
 \begin{eqnarray*}
&& \mathbb{E}\left[\left(\sum_{f_1\in\mathcal{G}_{\Omega,\delta}^D}\psi(f_1)h_{\delta}(f_1)A(f_1)\right)\cdot...\cdot\left(\sum_{f_k\in\mathcal{G}_{\Omega,\delta}^D}\psi(f_k)h_{\delta}(f_k)A(f_k)\right)\right]\\
&=&\sum_{f_1,...,f_k\in\mathcal{G}_{\Omega,\delta}^D}\psi(f_1)A(f_1)...\psi(f_k)A(f_k)\mathbb{E}\left[h_{\delta}(f_1)\cdot...h_{\delta}(f_k)\right]\\
&\sim&\int_{\Omega}...\int_{\Omega}\mathbb{E}[\mathcal{F}(z_1)\cdot...\cdot\mathcal{F}(z_k)]\prod_{i=1}^{k}\psi_i(z_i)dxdy+O(\delta),
 \end{eqnarray*}
 That is, the moments of left side of (\ref{wkc}) converges to those of right side of (\ref{wkc}).
 
To verify the last identity,  we notice that if $\ell$ is the number of coinciding $z_i's$, by Lemma \ref{sg}, the convergence always holds. Given the error term $O(\delta)$, the characteristic function of the left hand side of (\ref{wkc}) converges to the characteristic function of the right hand side of (\ref{wkc}). Since the right hand side of (\ref{wkc}) is a Gaussian random variable, its characteristic function is always continuous at 0. Hence the weak convergence (\ref{wkc}) is true, and the theorem follows.
 \end{proof}
 
\noindent\textbf{Acknowlegements} The work was supported by the Engineering and Physical Sciences Research Council under grant EP/103372X/1.
 
 \bibliography{dimiso}
\bibliographystyle{amsplain}

\end{document}